\documentclass[a4paper,11pt]{article}
\usepackage[utf8x]{inputenc}

\usepackage{parskip}
\usepackage{a4wide}
\usepackage{amssymb}
\usepackage{amsmath}
\usepackage{amsthm}
\usepackage[mathscr]{euscript}
\usepackage[normalem]{ulem}
\usepackage[colorlinks=true]{hyperref}
\usepackage[usenames,dvipsnames]{xcolor}
\usepackage{appendix}
\usepackage{tikz}
\usepackage{cleveref}
\usetikzlibrary{arrows,angles}
\usetikzlibrary{matrix}
\usetikzlibrary{decorations}
\usetikzlibrary{arrows,calc,shapes,decorations.pathreplacing}
\usepackage{tikz-cd}
\usepackage{todonotes}
\usepackage{mathtools}
\usepackage{mathrsfs}

\hypersetup{hidelinks}

\numberwithin{equation}{section}

\theoremstyle{definition}

\newtheorem{Definition}{Definition}[section]

\newtheorem{Remark}[Definition]{Remark}
\newtheorem{Example}[Definition]{Example}

\theoremstyle{plain}

\newtheorem{Theorem}[Definition]{Theorem}
\newtheorem{Proposition}[Definition]{Proposition}
\newtheorem{Lemma}[Definition]{Lemma}
\newtheorem{Corollary}[Definition]{Corollary}

\newcommand{\R}{\mathbb R}
\newcommand{\N}{\mathbb N}

\newcommand{\spn}{\mathrm{span}}

\usepackage[xcolor]{changebar}
\cbcolor{blue}

\DeclareMathOperator{\core}{\text{core}}

\DeclareMathOperator{\Int}{Int}

\usepackage[inline]{enumitem}
\newcommand{\enumlabelformat}{\roman}
\newcommand{\enumlabelfont}[1]{#1}
\newlength{\thelabelsep}
\setlist{labelsep=\thelabelsep}
\setlist[enumerate,1]{font=\enumlabelfont,label=(\enumlabelformat*),leftmargin=2.5em}
\setlist[itemize]{leftmargin=2.5em,label=$-$}

\newcounter{inlineenum}
\renewcommand{\theinlineenum}{\enumlabelformat{inlineenum}}
\let\epsilon\varepsilon
\let\phi\varphi

\makeatletter
\let\save@mathaccent\mathaccent
\newcommand*\if@single[3]{%
  \setbox0\hbox{${\mathaccent"0362{#1}}^H$}%
  \setbox2\hbox{${\mathaccent"0362{\kern0pt#1}}^H$}%
  \ifdim\ht0=\ht2 #3\else #2\fi
  }
\newcommand*\rel@kern[1]{\kern#1\dimexpr\macc@kerna}
\newcommand*\widebar[1]{\@ifnextchar^{{\wide@bar{#1}{0}}}{\wide@bar{#1}{1}}}
\newcommand*\wide@bar[2]{\if@single{#1}{\wide@bar@{#1}{#2}{1}}{\wide@bar@{#1}{#2}{2}}}
\newcommand*\wide@bar@[3]{%
  \begingroup
  \def\mathaccent##1##2{%
    \let\mathaccent\save@mathaccent
    \if#32 \let\macc@nucleus\first@char \fi
    \setbox\z@\hbox{$\macc@style{\macc@nucleus}_{}$}%
    \setbox\tw@\hbox{$\macc@style{\macc@nucleus}{}_{}$}%
    \dimen@\wd\tw@
    \advance\dimen@-\wd\z@
    \divide\dimen@ 3
    \@tempdima\wd\tw@
    \advance\@tempdima-\scriptspace
    \divide\@tempdima 10
    \advance\dimen@-\@tempdima
    \ifdim\dimen@>\z@ \dimen@0pt\fi
    \rel@kern{0.6}\kern-\dimen@
    \if#31
      \overline{\rel@kern{-0.6}\kern\dimen@\macc@nucleus\rel@kern{0.4}\kern\dimen@}%
      \advance\dimen@0.4\dimexpr\macc@kerna
      \let\final@kern#2%
      \ifdim\dimen@<\z@ \let\final@kern1\fi
      \if\final@kern1 \kern-\dimen@\fi
    \else
      \overline{\rel@kern{-0.6}\kern\dimen@#1}%
    \fi
  }%
  \macc@depth\@ne
  \let\math@bgroup\@empty \let\math@egroup\macc@set@skewchar
  \mathsurround\z@ \frozen@everymath{\mathgroup\macc@group\relax}%
  \macc@set@skewchar\relax
  \let\mathaccentV\macc@nested@a
  \if#31
    \macc@nested@a\relax111{#1}%
  \else
    \def\gobble@till@marker##1\endmarker{}%
    \futurelet\first@char\gobble@till@marker#1\endmarker
    \ifcat\noexpand\first@char A\else
      \def\first@char{}%
    \fi
    \macc@nested@a\relax111{\first@char}%
  \fi
  \endgroup
}
\makeatother
\newcommand\restr[2]{{% we make the whole thing an ordinary symbol
  \left.\kern-\nulldelimiterspace % automatically resize the bar with \right
  #1 % the function
  \vphantom{\big|} % pretend it's a little taller at normal size
  \right|_{#2} % this is the delimiter
  }}

\title{Topological cones and positively polarizable hyperbolic norms}

\author{Ethan Kharitonov\thanks{{\tt ethan.kharitonov@mail.utoronto.ca}, Department of Mathematics, University of Toronto, 40 St. George Street, M5S 2E4 Toronto, Ontario, Canada.}\\Argam Ohanyan \thanks{{\tt argam.ohanyan@utoronto.ca}, Department of Mathematics, University of Toronto, 45 St. George Street, M5S 2E5 Toronto, Ontario, Canada.}}

\newcommand{\inner}[2]{\langle #1, #2 \rangle}
\newcommand{\norm}[1]{\| #1\|}

\begin{document}

\date{\today}

%\date{Received: date /Accepted: date}

\maketitle

\begin{abstract}

In the first part of this article, we study linear cones over totally ordered fields. We show that for each such cone there uniquely exists a universal vector space (called its spanned vector space) into which it embeds as a generating convex cone. Moreover, we investigate topologies on cones for which the natural cone operations are continuous, and study how these topologies carry over to the spanned vector space. In the second part, we deal with hyperbolic norms which satisfy a polarization identity and are defined on cones over the real numbers. We show that, under reasonable assumptions, such hyperbolic norms induce a Lorentzian inner product on the spanned vector space. Finally, we establish a link between completeness under the Wick rotation of a Lorentzian inner product and order-theoretic completeness.

\vspace{1em}

\noindent
\emph{Keywords:} Hyperbolic norm, Lorentzian vector space, linear cone
\medskip

\noindent
\emph{MSC2020:} 
46A40, 52A07, 57N17, 53B30
\end{abstract}
\tableofcontents

\section{Introduction}

The framework to study Lorentzian and spacetime geometry has recently undergone a dramatic expansion. Considerations from (mathematical) general relativity necessitate a detailed analysis of more general Lorentzian structures than just smooth metric tensors on manifolds. We refer to the excellent recent survey by Braun \cite{braun2025new} and the references therein for an overview of the recent developments in this direction.

Our focus of study in this work are hyperbolic norms, which were recently introduced by Beran, Braun, Calisti, Gigli, McCann, Rott, Sämann and the second author in \cite[App.\ A.3]{beran2024nonlinear} as a Lorentzian analogue for the usual notion of norm which arises in the context of positive definite inner products on vector spaces. In that work, hyperbolic norms are defined (roughly) as $1$-homogeneous nonnegative functions on a (real) vector space satisfying the reverse triangle inequality. Notably, the hyperbolic norm needs to be set to $-\infty$ outside of its conic nonnegativity set to guarantee the reverse triangle inequality everywhere. The authors then go on to show that, under reasonable assumptions, such hyperbolic norms (on a finite-dimensional vector space $V$ with $\dim V = n$) lead to globally hyperbolic metric spacetime structures \cite[Prop.\ A.10]{beran2024nonlinear} which satisfy the timelike curvature-dimension condition $\mathsf{TCD}_q(0,n)$ for every $0 \neq q < 1$ \cite[Prop.\ A.11]{beran2024nonlinear}. Further, the authors introduce the notion of positive polarizability for hyperbolic norms \cite[Eq.\ (A.14)]{beran2024nonlinear}, which is akin to a parallelogram identity, and show that positively polarizable hyperbolic norms always induce a bilinear pairing on their future (nonnegativity) cone which extends uniquely to an indefinite inner product on all of $V$ \cite[Lem.\ A.12]{beran2024nonlinear}. Moreover, \cite[Lem.\ A.13]{beran2024nonlinear} clarifies that the triangle inequality resp.\ reverse triangle inequality of a norm arising from a nondegenerate inner product on a finite-dimensional vector space characterizes that it is in fact positive definite resp.\ Lorentzian. 

Since the introduction in \cite{beran2024nonlinear}, hyperbolic norms have been further studied by Gigli: In his comprehensive recent work \cite{gigli2025hyperbolic}, Gigli studies directed completions in the general setting of ordered spaces and develops the functional-analytic theory of hyperbolic Banach spaces.

Let us also mention Minguzzi's \cite{minguzzi2019causality} earlier work on the theory of closed cone structures on manifolds, which already identified the natural conic nature of the domains of definition of (low regularity) Lorentz--Finsler structures.

Our aim in this article is to continue the investigation into fundamental properties of hyperbolic norms, picking up where \cite[App.\ A.3]{beran2024nonlinear} left off. It appears to be natural to study hyperbolic norms which are defined on linear cones, rather than vector spaces (a point of view coming from Gigli \cite{gigli2025hyperbolic}). To this end, we begin in Section \ref{Section: linearcones} by studying linear cones $F$ over an arbitrary totally ordered field $\mathbb{K}$. This means that $F$ is a cancellative, Abelian monoid for which scalar multiplication with nonnegative elements of $\mathbb{K}$ is defined, subject to usual compatibility axioms. It is well-known (see e.g.\ the textbook of Bruns--Gubeladze \cite[p.\ 50]{BrunsGubeladze}) that every cancellative Abelian monoid embeds into an Abelian group which extends the monoid in a universal way, known as its Grothendieck group. We show that the Grothendieck group of a cone $F$ over a totally ordered field becomes a $\mathbb{K}$-vector space satisfying a universal property in a natural way. We call that vector space the span of $F$ and denote it by $\spn(F)$. Our results in this context are Lemma \ref{Lemma: Uniqueness of spans} and Theorem \ref{Theorem: existenceofspans}.

We continue with the study of topological cones, i.e., cones such that the natural operations (addition and nonnegative scalar multiplication) are continuous (here, a totally ordered field is topologized via the order topology). We show the existence of a universal topology on $\spn(F)$ with which it becomes a topological vector space satisfying a universal property, see Theorem \ref{prop: universal topology of span}. In the rest of Section \ref{Section: linearcones}, we compare the universal topology with other (natural) choices of topologies on $\spn(F)$. In particular, we show that if $F$ carries the subspace topology of some norm topology on $\spn(F)$ and has nonempty interior in $\spn(F)$, then many of these natural choices of topologies on $\spn(F)$ coincide (see Theorem \ref{Theorem: normed cones quotient topology}).

We then move on to investigate positively polarizable hyperbolic norms defined on linear cones over $\R$ in Section \ref{Section: Positively polarizable hyperbolic norms}. In Theorem \ref{Theorem: innerproductnondegenerate} we show that, under natural conditions on the positively polarizable hyperbolic norm (finiteness, non-triviality and strict hyperbolicity), the induced inner product on $\spn(F)$ is non-degenerate. Moreover, in Theorem \ref{Theorem: innerproductLorentzian}, we give three sufficient conditions under which the inner product is Lorentzian. Let us remark that we do not assume $\spn(F)$ to be finite-dimensional in any of these results. Finally, in Theorem \ref{Theorem: Sequential foward completeness and Lorentz Hilbert}, we establish the equivalence between order completeness of $\spn(F)$ (where the order on $\spn(F)$ is defined by $x \leq y$ if and only if $y - x \in F$) and completeness under the (positive definite) Wick rotation of the Lorentzian inner product induced by the positively polarizable hyperbolic norm.

In Section \ref{Section: Outlook}, we summarize our work and describe some open problems related to the study of hyperbolic norms.

\section{Topologies on linear cones}\label{Section: linearcones}

\subsection{Linear cones and their spanned vector spaces}

\begin{Definition}[Totally ordered field]
A \emph{totally ordered field} is a field $\mathbb{K}$ together with a total ordering $\leq$ (i.e., $\leq$ is a partial order and for any two elements $\lambda, \mu \in \mathbb{K}$ either $\lambda > \mu$, $\lambda = \mu$, or $\lambda < \mu$; here $<$ means $\leq$ but not $=$) that is compatible with the field operations on $\mathbb K$:
\begin{enumerate}
    \item[(i)] If $\mu \leq \lambda$, then for any $\nu \in \mathbb{K}$ also $\mu + \nu \leq \lambda + \nu$.
    \item[(ii)] If $\lambda, \mu \geq 0$, then also $\lambda \mu \geq 0$.
\end{enumerate}
\end{Definition}

It is easily observed that a totally ordered field $\mathbb{K}$ has characteristic $0$. In particular, there is an order preserving injective homomorphism of fields $\mathbb{Q} \hookrightarrow \mathbb{K}$, so $\mathbb{Q}$ can be understood to be a subfield of $\mathbb{K}$.

\begin{Remark}[Non-negative and non-positive elements]
    Every totally ordered field $\mathbb{K}$ decomposes as $\mathbb{K}_+ \cup \mathbb{K}_-$, where $\mathbb{K}_+$ consists of the nonnegative elements $\lambda \geq 0$ and $\mathbb{K}_-$ consists of the nonpositive elements $\mu \leq 0$, with $\mathbb{K}_+ \cap \mathbb{K}_- = \{0\}$. The set of nonnegative elements $\mathbb{K}_+$ is a commutative semigroup with addition, and $\mathbb{K}_+ \setminus \{0\}$ is a commutative group with multiplication.
\end{Remark}

\begin{Definition}[Linear cone]
Let $(\mathbb{K},\leq)$ be a totally ordered field. A \emph{linear cone over} $\mathbb{K}$ is a commutative, cancellative monoid $(F,+)$ together with a \emph{scalar multiplication} $\mathbb{K}_+ \times F \to F$ such that for all $\lambda, \mu \in \mathbb{K}_+$, $v,w \in F$:
\begin{align}
    &\lambda(v + w) = \lambda v + \lambda w,\\
    &(\lambda + \mu)v = \lambda v + \mu v,\\
    &\lambda (\mu v) = (\lambda \mu)v,\\
    &1v = v.
\end{align}
Cancellativity automatically yields $0v = 0$ for all $v \in F$. We call a linear cone $F$ \emph{proper} (or \emph{pointed at} $0 \in F$) if it contains no negatives, i.e., no element $v \in F\setminus \{0\}$ has a negative in $F$.
\end{Definition}

\begin{Definition}[Linear maps]
Let $F$ be a linear cone over a totally ordered field $\mathbb{K}$. Let $W$ either be a linear cone over $\mathbb{K}$ or a $\mathbb{K}$-vector space. A map $f: F \to W$ is called $\mathbb{K}_+$-\emph{linear} if for all $\lambda, \mu \in \mathbb{K}_+$ and all $v,w \in F$
\begin{align}
    f(\lambda v + \mu w) = \lambda f(v) + \mu f(w).
\end{align}
\end{Definition}
There is a well-known algebraic construction that embeds a commutative, cancellative monoid in an Abelian group, called its \textit{Grothendieck group}, in a minimal way (see e.g.\ the discussion in \cite[p.50]{BrunsGubeladze}). We show that if a commutative cancellative monoid $F$ has the additional structure of a linear cone over a totally ordered field, then its Grothendieck group is a vector space in a natural way. For convenience, we give all of the details of the construction and do not rely on knowledge about Grothendieck groups.
\begin{Definition}[Span of a linear cone]
Let $F$ be a linear cone over a totally ordered field $\mathbb{K}$. We say that a $\mathbb{K}$-vector space $V$ together with a $\mathbb{K}_+$-linear map $\iota:F \to V$ is a \emph{span of} $F$ if it satisfies the following universal property: For every $\mathbb{K}$-vector space $W$ and every $\mathbb{K}_+$-linear map $f:F \to W$, there exists a unique $\mathbb{K}$-linear map $\Tilde{f}:V \to W$ such that $f = \Tilde{f} \circ \iota$. In the form of a commutative diagram, this reads:
\[
\begin{tikzcd}
F \arrow[dr, "\iota"'] \arrow[rr, "\forall \, f"] & & W \\
& V \arrow[ur, "\exists! \, \tilde{f}"'] &
\end{tikzcd}
\]
\end{Definition}

\begin{Lemma}[Uniqueness of spans]\label{Lemma: Uniqueness of spans}
Let $F$ be a linear cone over a totally ordered field $\mathbb{K}$. Then there exists at most one span of $F$ (up to unique isomorphism of $\mathbb{K}$-vector spaces).
\end{Lemma}
\begin{proof} This is a standard argument using the universal property.
\end{proof}

\begin{Theorem}[Existence of spans]\label{Theorem: existenceofspans}
Let $F$ be a linear cone over a totally ordered field $\mathbb{K}$. Then there exists a unique span of $F$, denoted $(\spn(F),\iota)$. Moreover, the map $\iota: F \to \spn(F)$ is injective, and the image $\iota(F) \subseteq \spn(F)$ is a generating convex cone in $\spn(F)$, and $\spn(F) = \{v - w: v,w \in \iota(F)\}$. If $F$ is proper, then so is $\iota(F)$.
\end{Theorem}
\begin{proof}
    We have already shown uniqueness in Lemma \ref{Lemma: Uniqueness of spans}. To construct $\spn(F)$, we proceed in a manner similar to how one constructs the integers from the natural numbers: On $F \times F$, consider the relation
    \begin{align}
        (v_1,w_1) \sim (v_2,w_2) \quad :\Leftrightarrow \quad v_1 + w_2 = v_2 + w_1.
    \end{align}
    It is elementary to check that this is an equivalence relation (for this, it is essential that $F$ be cancellative). Define $\spn(F):=(F \times F) / \sim$. Denote the equivalence class of $(v,w)$ by $v-w$, with $v$ the equivalence class of $(v,0)$ and $-v$ the equivalence class of $(0,v)$. Then $\spn(F)$ becomes a $\mathbb{K}$-vector space with the addition defined by $(v_1 - w_1) + (v_2 - w_2) := (v_1 + v_2) - (w_1 + w_2)$ and scalar multiplication defined by $\lambda (v - w):= \lambda v - \lambda w$ for $\lambda \geq 0$, and $(-\lambda)(v - w):= \lambda w - \lambda v$ for $-\lambda \leq 0$. It is easy to check that $\spn(F)$ is a $\mathbb{K}$-vector space. The map $\iota:F \to \spn(F)$ sends $v \in F$ to $v \in \spn(F)$, i.e., the equivalence class of $(v,0)$, and is clearly injective. Moreover, since any element of $\spn(F)$ is of the form $v-w$ with $v,w \in F$, the span of $\iota(F)$ is all of $\spn(F)$. Let us now show that $(\spn(F),\iota)$ is a span of $F$: Given any $\mathbb{K}_+$-linear map $f:F \to W$, with $W$ a $\mathbb{K}$-vector space, its unique $\mathbb{K}$-linear extension $\Tilde{f}: \spn(F) \to W$ is given by $\Tilde{f}(v - w):= f(v) - f(w)$. Let us check that $\tilde f$ is well-defined: Indeed, if $v_1 + w_2 = v_2 + w_1$, then $ f(v_1) + f(w_2)=f(v_1 + w_2)  = f(v_2+ w_1) = f(v_2) + f(w_1)$ by linearity, hence $f(v_1) - f(w_1) = f(v_2) - f(w_2).$ This shows that $(\spn(F),\iota)$ satisfies the universal property, concluding the proof.
\end{proof}

\subsection{Topologies on linear cones and their span}

\begin{Remark}[Order topology on totally ordered fields]
Let $(\mathbb{K},\leq)$ be a totally ordered field. Then there is a natural order topology on $\mathbb{K}$ generated by the basis of open intervals $(a,b)$, $(-\infty,a)$ and $(a,\infty)$, for $a,b \in \mathbb{K}$. With this topology, all of the field operations $+:\mathbb{K}^2 \to \mathbb{K}$, $-:\mathbb{K}^2 \to \mathbb{K}$, $^{-1}:\mathbb{K} \setminus \{0\} \to \mathbb{K} \setminus \{0\}, \cdot:\mathbb{K}^2 \to \mathbb{K}$ are continuous, i.e., $\mathbb{K}$ becomes a topological field. Whenever we speak of a totally ordered field, we will understand it to be equipped with the order topology.
\end{Remark}

\begin{Definition}[Topological linear cone]
A \textit{topological linear cone} over a totally ordered field $\mathbb{K}$ is a linear cone $F$ over $\mathbb{K}$, equipped with a topology such that addition $+:F \times F \to F$ and nonnegative scalar multiplication $\cdot: \mathbb{K}_{+} \times F \to F$ are continuous.
\end{Definition}

\begin{Lemma}[Existence of a solution set]\label{Lemma: existence of solution set}
    Let $F$ be a topological linear cone over a totally ordered field $\mathbb{K}$. There exists a set $S_{F} = (W_{i}, f_{i})_{i \in I}$ such that:
    \begin{enumerate}
        \item Each $W_{i}$ is a topological vector space over $\mathbb{K}$.
        \item Each $f_{i}: F \longrightarrow W_{i}$ is a continuous $\mathbb{K}_{+}$-linear map.
        \item For any topological vector space $W$ over $\mathbb{K}$ and any continuous $\mathbb{K}_{+}$-linear map $f: F \longrightarrow W$, there exists some $i \in I$ and a continuous $\mathbb{K}$-linear map $h: W_{i} \longrightarrow W$ such that $f = h \circ f_{i}$.
    \end{enumerate}
    Such a set $S_{F}$ is called a \textit{solution set} for $F$.
\end{Lemma}

\begin{proof}
    Note first that we cannot take $S_{F}$ to be the collection of all topological vector spaces $W$ along with continuous $\mathbb{K}_{+}$-linear maps $f: F \longrightarrow W$ because this would be a proper class.
    
    Let $\kappa = |F|$ be the cardinality of $F$. Then $|\spn(F)| \leq \kappa^{2}$ (or is exactly $\kappa$ if $F$ is infinite). Let $\mu = |\spn(F)|$. For each cardinal $\alpha \leq \mu$ fix a set $U_{\alpha}$ of cardinality $\alpha$ (a choice of $U_{\alpha} = \alpha$ works). Let $\mathcal{TVS}(U_{\alpha})$ denote the collection of all topological vector spaces whose underlying set is $U_{\alpha}$. We know that $\mathcal{TVS}(U_{\alpha})$ is a set because the number of possible topologies on $U_{\alpha}$ is bounded by $2^{2^{\alpha}}$ and the number of possible linear structures making $U_{\alpha}$ into a topological vector space is also bounded in terms of $\alpha$. Any topological vector space whose underlying set has cardinality $\alpha$ is isomorphic to a space in $\mathcal{TVS}(U_{\alpha})$. Define
    $$\mathcal{W}_{rep} = \bigcup_{\alpha \leq \mu}\mathcal{TVS}(U_{\alpha}),$$
    where the union ranges over all cardinals $\alpha \leq \mu$ which form a subset of $\mu + 1$. It follows that $\mathcal{W}_{rep}$ is also a set. Next, let
    $$S_{F} = \{(W, f): W \in \mathcal{W}_{rep},\; f: F \longrightarrow W \text{ continuous and } \mathbb{K}_{+}\text{-linear}\}.$$
    Note that $S_{F}$ is also a set. Now we show that $S_{F}$ is in fact the desired solution set. Given a topological vector space $W$ and a continuous $\mathbb{K}_{+}$-linear map $f: F \longrightarrow W$, let $\tilde{f}: \spn(F) \longrightarrow W$ be its unique $\mathbb{K}$-linear extension. The image $\tilde{f}(\spn(F))$ is a subspace of $W$ and when given the subspace topology becomes a topological vector space. Since the cardinality of $\tilde{f}(\spn(F))$ is at most $\mu$, $\tilde{f}(\spn(F))$ is isomorphic, as a topological vector space, to some $W' \in \mathcal{W}_{rep}$. Let $\phi: \tilde{f}(\spn(F)) \longrightarrow W'$ be an isomorphism. Since $f$ maps $F$ into the domain of $\phi$, that is, $f(F) \subseteq \tilde{f}(\spn(F))$, we may define $f' = \phi \circ f: F \longrightarrow W'$. Since both $f$ and $\phi$ are $\mathbb{K}_{+}$-linear and continuous, by definition $(W', f') \in S_{F}$, and letting $j: \tilde{f}(\spn(F)) \longrightarrow W$ be the inclusion map, we get 
    $$f = j \circ f = (j \circ \phi^{-1}) \circ f'.$$
    Here $j \circ \phi^{-1}$ is continuous and $\mathbb{K}$-linear as desired.
\end{proof}

\begin{Theorem}[Construction of the universal topology]\label{prop: universal topology of span}
    Let $F$ be a topological linear cone over $\mathbb{K}$. There exists a topology on $\spn(F)$ such that
    \begin{enumerate}
        \item $\spn(F)$ is a topological vector space.
        \item $\iota: F \longrightarrow \spn(F)$ is continuous.
        \item For any topological vector space $W$ and continuous $\mathbb{K}_{+}$-linear map $f: F \longrightarrow W$, the unique $\mathbb{K}$-linear map $\tilde{f}: \spn(F) \longrightarrow W$ such that $f = \tilde{f} \circ \iota$ is continuous.
    \end{enumerate}
\end{Theorem}
\begin{proof}
    Let $S_{F} = (W_{i}, f_{i})_{i \in I}$ be the solution set constructed in Lemma $\ref{Lemma: existence of solution set}$. Define 
    $$P = \prod_{i \in I}W_{i}.$$
   % Note that $P$ is nonempty by the Axiom of Choice since $(W_{i})_{I}$ is a set. 
   
   This is a topological vector space under the product topology. Let $\tilde{f_{i}}: \spn(F) \longrightarrow W_{i}$ be the unique $\mathbb{K}$-linear extension of $f_{i}$ and define the map $E: \spn(F) \longrightarrow P$ by
    $$E(x) = (\tilde{f_{i}}(x))_{i \in I}.$$
    Since each $\tilde{f_{i}}$ is $\mathbb{K}$-linear, so is $E$. We endow $\spn(F)$ with the initial topology induced by the map $E$ (this is the same as the initial topology with respect to the family of maps $\{\tilde{f}_i: \spn(F) \to W_i\}_{i \in I}$). Let $+: \spn(F)^{2} \longrightarrow \spn(F)$ denote addition in $\spn(F)$ and $+_{P}: P^{2} \longrightarrow P$ denote addition in $P$. By the characteristic property of the initial topology, $+$ is continuous if and only if $E \circ +$ is continuous. By linearity of $E$, $E(x + y) = E(x) +_{P} E(y)$ for all $x, y \in \spn(F)$. That is,
    $$E \circ + = +_{P} \circ (E \times E).$$
    The functions $E \times E$ and $+_{P}$ are continuous so we conclude that $+$ is also continuous. Similarly, let $M: \mathbb{K} \times \spn(F) \longrightarrow \spn(F)$ and $M_{P}: \mathbb{K} \times P \longrightarrow P$ denote scalar multiplication in $\spn(F)$ and $P$ respectively. Recall that $\mathbb{K}$ is a topological field and $\mathbb{K} \times \spn(F)$ is given the product topology. By the characteristic property of the initial topology on $\spn(F)$, the map $M$ is continuous if and only if $E \circ M$ is continuous. Again, by linearity of $E$,
    $$E \circ M = M_{P} \circ (id_{\mathbb{K}} \times E).$$
    The functions $M_{P}$, $id_{\mathbb{K}}$ and $E$ are all continuous so we conclude that $M$ is continuous as well. This establishes \textit{(i)}. To show \textit{(ii)}, observe that the inclusion map $\iota$ is continuous if and only if $E \circ \iota$ is continuous. However,
    $$E \circ \iota = (\tilde{f_{i}} \circ \iota)_{i \in I} = (f_{i})_{i \in I}.$$
    Since each component $f_{i}$ of $E \circ \iota$ is continuous, so is $E \circ \iota$. Finally, we show that $\spn(F)$ satisfies the universal property of \textit{(iii)}. Let $W$ be any topological vector space and let $f: F \longrightarrow W$ be any continuous $\mathbb{K}_{+}$-linear map. Let $\tilde{f}: \spn(F) \longrightarrow W$ be its unique $\mathbb{K}$-linear extension. We aim to show that $\tilde{f}$ is continuous. By definition of $S_{F}$, there exists an $i \in I$ and a continuous $\mathbb{K}$-linear map $h: W_{i} \longrightarrow W$ such that $f = h \circ f_{i}$. Let $\tilde{f}_{i}: \spn(F) \longrightarrow W_{i}$ be the unique $\mathbb{K}$-linear extension of $f_{i}$ to $\spn(F)$. Then,
    $$(h \circ \tilde{f_{i}}) \circ \iota = h \circ (\tilde{f_{i}} \circ \iota) = h \circ f_{i} = f = \tilde{f} \circ \iota.$$
    Pictorially, the following diagram commutes.
    $$
        \begin{tikzcd}[row sep=2cm, column sep=2.8cm] 
          F \arrow[d, "\iota"'] 
            \arrow[dr, "f_i", pos=0.6] 
            \arrow[drr, "f", pos=0.4] 
          & & \\
          \text{span}(F) \arrow[r, "\tilde{f_i}"'] 
                         \arrow[rr, "\tilde{f}"', bend right=25] 
          & W_i \arrow[r, "h"'] 
          & W 
        \end{tikzcd}
    $$
    Since both $h \circ \tilde{f_{i}}$ and $\tilde{f}$ are an extension of $f$ as $\mathbb{K}$-linear maps $\spn(F) \to W$, by the uniqueness portion of the definition of $\spn(F)$,
    $$h \circ \tilde{f_{i}} = \tilde{f}.$$
    Let $\pi_{i}: P \longrightarrow W_{i}$ be the projection maps. By definition of $E$, we have $\tilde{f_{i}} = \pi_{i} \circ E$ so
    $$\tilde{f} = h \circ \tilde{f_{i}} = h \circ \pi_{i} \circ E.$$
    Since $h$, $\pi_{i}$ and $E$ are all continuous, so is $\tilde{f}$.
\end{proof}

\begin{Definition}[Universal topology]
Let $F$ be a topological linear cone over a totally ordered field $\mathbb{K}$. We call the topology on $\spn(F)$ constructed in the previous result the \emph{universal topology}, and denote it by $\tau_{univ}$. Note that no choices were involved in its construction aside from the choice of sets of every cardinality below that of $\spn(F)$. It is easily observed that $\tau_{univ}$ does not depend on this choice.
\end{Definition}

\begin{Remark}[Interpretation in the language of reflective subcategories]
    Let $\mathbb{K}$ be a totally ordered field. Let $TVS_{\mathbb{K}}$ be the category of topological vector spaces over $\mathbb{K}$ with continuous linear maps and let $TopCone_{\mathbb{K}}$ be the category of topological linear cones over $\mathbb{K}$ with continuous $\mathbb{K}_{+}$-linear maps. Observe that every TVS is trivially a topological cone (simply forget subtraction), and similarly, every $TVS_{\mathbb{K}}$ morphism is a $TopCone_{\mathbb{K}}$ morphism (with the same domain and range). That is $TVS_{\mathbb{K}}$ is a subcategory of $TopCone_{\mathbb{K}}$. Further, it is a full subcategory because every $TopCone_{\mathbb{K}}$ morphism between two objects of $TVS_{\mathbb{K}} \subseteq TopCone_{\mathbb{K}}$ is automatically a $TVS_{\mathbb{K}}$ morphism. In the language of category theory, Theorem \ref{prop: universal topology of span} can be interpreted as the statement that $TVS_{\mathbb{K}}$ is a reflective subcategory of $TopCone_{\mathbb{K}}$, and that $(\spn(F),\iota)$ is the $TVS_{\mathbb{K}}$-reflection of a cone $F$. Equivalently, we may define a functor $R: TopCone_{\mathbb{K}} \longrightarrow TVS_{\mathbb{K}}$ by
    $R(F) = (\spn(F), \tau_{univ})$ for any topological cone $F$ and $R(f) = \tilde{f}:\spn(F) \to \spn(G)$ for any continuous $\mathbb{K}_{+}$-linear map $f: F \longrightarrow G$ where $G \in TopCone_{\mathbb{K}}$. Then $R$ is the reflector, i.e., the left adjoint of the inclusion (i.e., forgetful) functor $U: TVS_{\mathbb{K}} \longrightarrow TopCone_{\mathbb{K}}$. The unit of this adjunction is the natural transformation $id_{TopCone_{\mathbb{K}}} \longrightarrow U\circ R$ whose components are the canonical maps $\iota$.

    It should be noted that our proof of Theorem \ref{prop: universal topology of span} is comparable to that of the Freyd adjoint functor theorem (see e.g.\ the textbook of MacLane \cite[Sec.\ V.6, Thm.\ 2]{MacLane}), where a solution set is used to construct the adjoint functor.
\end{Remark}

The following is an immediate consequence of Theorem \ref{prop: universal topology of span}.

\begin{Corollary}\label{Corollary: univeral topology finest}
    Suppose $F$ is a topological linear cone over $\mathbb{K}$ and $\spn(F)$ is equipped with a topology $\tau$. If
    \begin{enumerate}
        \item $(\spn(F), \tau)$ is a topological vector space,
        \item the canonical map $\iota: F \longrightarrow (\spn(F), \tau)$ is continuous,
    \end{enumerate}
    then $\tau \subseteq \tau_{univ}$.
\end{Corollary}

The next proposition, i.e., Proposition \ref{Prop: iota embedding char univ top} below, says that if we upgrade $\iota$ from being merely continuous to being a topological embedding, then the two topologies become equal provided $\iota(F)$ has nonempty interior in $\spn(F)$. This means that, assuming that there exists some topology for which $\iota$ is an embedding, it is in fact the universal topology $\tau_{univ}$. Before proving it, we discuss some basic properties of topological vector spaces over $\mathbb{K}$. 

\begin{Remark}[Standard properties of topological vector spaces over totally ordered fields]\label{Rem: Standard properties of TVSs}
    Suppose $V$ is a TVS over some totally ordered field $\mathbb{K}$. 
    We make use of the following standard facts. 
    \begin{enumerate}
        \item The field $\mathbb{K}$ has characteristic $0$ because it is totally ordered. Thus it makes sense to divide by any element of $\mathbb{Q}\setminus\{0\}$.
        \item Every neighborhood $N$ of $0$ is absorbing. To see this, fix any $x \in V$ and consider the map $\phi_{x}: \mathbb{K} \longrightarrow V$ given by $\phi_{x}(\lambda) = \lambda x$. Since $\phi_{x}$ is continuous, the set $\phi_{x}^{-1}(N)$ is a neighborhood of $0$. Since $\mathbb{K}$ is given the order topology, there exists a $0 < \delta \in \mathbb{K}$ such that $(-\delta, \delta) \subseteq \phi_{x}^{-1}(N)$. Thus, $\frac{1}{2}\delta x = \phi_{x}(\frac{1}{2} \delta) \in N$. 
        \item If $N \subseteq V$ with $\Int(N) \neq \emptyset$ then $N - N$ is a neighborhood of $0$. 
        \item If $U \subseteq V$ is an open set containing $0$ then there exists an open set $B$, also containing $0$ such that $B - B \subseteq U$.
    \end{enumerate}
\end{Remark}

\begin{Lemma}\label{Lemma: Image has non empty interior}
    Let $F$ be a topological linear cone and suppose that $\spn(F)$ is given some topology $\tau$ under which it is a TVS. If the inclusion map $\iota$ is an embedding and $\iota(F)$ has a non-empty interior in $\spn(F)$, then for any open $N \subseteq F$ containing $0$, the image $\iota(N)$ has non-empty interior in $\spn(F)$.
\end{Lemma}
\begin{proof}
    Since $\iota$ is an embedding and $N$ is open, $\iota(N)$ is open in $\iota(F)$. Further, $0 \in \iota(N)$ since $\iota(0) = 0$. Thus, there exists an $M$, open in $\spn(F)$ such that $M \cap \iota(F) = \iota(N)$. Pick $s \in \Int(\iota(F))$. Since $\iota(N) \subseteq M$, we have $0 \in M$ so $M$ is absorbing. Let $0 < \delta \in \mathbb{K}$ such that $\delta s \in M$. Since scaling by $\delta$ is a homeomorphism of $\spn(F)$ with itself, $\delta s$ is in the interior of $\delta \cdot \iota(F) = \iota(F)$. Since $M$ is open, 
    \begin{equation*}
        \delta s \in M \cap \Int(\iota(F)) = \Int(M \cap \iota(F)) = \Int(\iota(N)).
    \end{equation*}
    Hence, $\iota(N)$ has a non-empty interior.
\end{proof}

\begin{Proposition}[$\iota$ being an embedding characterizes $\tau_{univ}$]
\label{Prop: iota embedding char univ top}
    Under the assumptions of Corollary $\ref{Corollary: univeral topology finest}$ and the additional assumptions that $\iota$ be an embedding and that $\iota(F)$ have a non-empty interior in $\spn(F)$, we get $\tau = \tau_{univ}$.
\end{Proposition}

\begin{proof}
    For notational simplicity throughout this proof, let $V_{univ} = (\spn(F), \tau_{univ})$ and $V_{\tau} = (\spn(F), \tau)$. We claim that the identity map $id_{\spn(F)}: V_{\tau} \longrightarrow V_{univ}$ is continuous. Recall that $\tau_{univ}$ is defined to be the initial topology induced by the functions $(W_{i}, \tilde{f_{i}})_{i \in I}$ where $S_{F} = (W_{i}, f_{i})_{i \in I}$ is a solution set. Hence it suffices to show that the map
    $$\tilde{f_{i}} = \tilde{f_{i}} \circ id_{\spn(F)}: V_{\tau} \longrightarrow W_{i}$$
    is continuous for each $i \in I$. We will show that for an arbitrary topological vector space $W$ and continuous $\mathbb{K}_{+}$-linear map $f: F \longrightarrow W$, the unique $\mathbb{K}$-linear extension $\tilde{f}: V_{\tau} \longrightarrow W$ is continuous. In particular this holds for any $(W, f) \in S_{F}$, yielding the claim. It suffices to prove that $\tilde{f}$ is continuous at $0 \in V_{\tau}$. Let $0 \in U_{W}$ be open in $W$. By part (\textit{iv}) of Remark \ref{Rem: Standard properties of TVSs}, there exists an open $0$-neighborhood $B_{W}$ in $W$ such that $B_{W} - B_{W} \subseteq U_{W}$. Since $f$ is continuous, the set $N_{F} = f^{-1}(B_{W})$ is open. Further, since $0 \in B_{W}$ and $f$ is $\mathbb{K}_{+}$-linear, $0 \in N_{F}$. Let
    \begin{equation*}
        S := \iota(N_{F}) - \iota(N_{F}).
    \end{equation*}
    Since $N_{F}$ is open, contains $0$, and by assumption, $\iota$ is an embedding with $\Int(\iota(F)) \neq \emptyset$, Lemma $\ref{Lemma: Image has non empty interior}$ implies that $\iota(N_{F})$ is a neighborhood in $V_{\tau}$. By part (\textit{iii}) of Remark \ref{Rem: Standard properties of TVSs}, $S$ is a $0$-neighborhood.
    
    Finally, we claim $\tilde{f}(S) \subseteq U_{W}$. If $y \in S$, then $y = y_{1} - y_{2}$ where $y_{1}, y_{2} \in \iota(N_{F})$. Thus, there exist $x_{1}, x_{2} \in N_{F}$ such that $\iota(x_{1}) = y_{1}$ and $\iota(x_{2}) = y_{2}$. Then,
    $$\tilde{f}(y) = \tilde{f}(y_{1}) - \tilde{f}(y_{2}) = \tilde{f}(\iota(x_{1})) - \tilde{f}(\iota(x_{2})) = f(x_{1}) - f(x_{2}).$$
    Since $x_{1}, x_{2} \in N_{F}$, we get $f(x_{1}), f(x_{2}) \in B_{W}$. Therefore, $\tilde{f}(y) \in B_{W} - B_{W} \subseteq U_{W}$
    so $\tilde{f}(S) \subseteq U_{W}$ so $\tilde{f}$ is continuous at $0$, which implies that it is continuous on all of $V_{\tau}$. This shows that $\tau$ is finer than the initial topology induced by $(W_{i}, \tilde{f}_{i})_{i \in I}$, that is $\tau_{univ} \subseteq \tau$.
\end{proof}

The following example shows that generating cones can in general have empty interior.

\begin{Example}
    Consider the vector space $\ell^1(\R)$ of real sequences $(a_n)_{n \geq 0}$ such that $\sum_n |a_n| < + \infty$. Then $C:=\{(a_n) \in \ell^1(\R) : a_n \geq 0 \quad \forall \, n\}$ is a linear cone with $\spn(C) = \ell^1(\R)$, but has empty interior when considered as a subspace of $\ell^1(\R)$ with the $1$-norm.
\end{Example}

There is another natural topology we can put on $\spn(F)$. That is the quotient topology $\tau_{quot}$ on $(F \times F)/\sim$. The next result tells us that in general, $\tau_{quot}$ is no better than $\tau_{univ}$ and that whenever $\tau_{quot}$ has any desirable properties it must actually be equal to $\tau_{univ}$.

\begin{Proposition}\label{prop: quot topology too fine}
    Let $F$ be a topological cone. Then $\tau_{univ} \subseteq \tau_{quot}$ with equality if and only if $(\spn(F), \tau_{quot})$ is a topological vector space.
\end{Proposition}
\begin{proof}
    To see that $\tau_{univ} \subseteq \tau_{quot}$, consider the quotient map $F \times F \to (\spn(F),\tau_{univ})$, $(x,y) \mapsto \iota(x) - \iota(y)$. As the composition of the (due to Theorem \ref{prop: universal topology of span} continuous) maps $F \times F \to \spn(F) \times \spn(F)$, $(x,y) \mapsto (\iota(x),\iota(y))$ and $\spn(F) \times \spn(F) \to \spn(F), (z,w) \mapsto z - w$, the quotient map is continuous. It descends to the identity map $(\spn(F),\tau_{quot}) \to (\spn(F),\tau_{univ})$, which is hence continuous as a consequence of the universal property of $\tau_{quot}$.
    
    If $(\spn(F),\tau_{quot})$ is also a TVS then Corollary $\ref{Corollary: univeral topology finest}$ implies that $\tau_{quot} \subseteq \tau_{univ}$ and consequently $\tau_{quot} = \tau_{univ}$. Conversely, if $\tau_{quot} = \tau_{univ}$, then $(\spn(F),\tau_{quot})$ is a TVS by Theorem \ref{prop: universal topology of span}(i).
\end{proof}

It would be desirable to identify conditions which guarantee that $(\spn(F), \tau_{quot})$ is a topological vector space.

\subsection{Norm topologies on cones}

In this subsection, we study cones whose topology is the subspace topology of a norm topology on $\spn(F)$. It turns out that the universal and quotient topologies on $\spn(F)$ in this case agree and are induced by a norm, provided the interior of $F$ in $\spn(F)$ is nonempty in the norm topology, see Theorem \ref{Theorem: normed cones quotient topology}.

\begin{Proposition}[Norm extension]
\label{Prop: norm extension}
    Let $F$ be a real linear cone. Suppose $\mathbf n$ is a norm on $\spn(F)$. Define the function $\tilde{\mathbf{n}}: \spn(F) \longrightarrow [0, \infty)$ by
    \begin{equation}
    \tilde{\mathbf{n}}(x) = \inf\{\mathbf{n}(u) + \mathbf{n}(v): u, v \in F; x = u - v\}.
    \end{equation}
    Then $\tilde{\mathbf{n}}$ is also a norm.
\end{Proposition}
\begin{proof}
    First, $\tilde{\mathbf{n}}$ is well-defined because $F$ is a generating cone so for any $x \in \spn(F)$ there are $u, v \in F$ with $x = u - v$. Now suppose $\tilde{\mathbf{n}}(x) = 0$ for some $x \in \spn(F)$. Whenever $x = u - v$ for $u,v \in F$, then $$\mathbf{n}(x) \leq \mathbf{n}(u) + \mathbf{n}(v)$$ by the triangle inequality for $\mathbf{n}$. Since $\tilde{\mathbf{n}}(x) = 0$, the right hand side can be made arbitrarily small via suitable choices of $u,v \in F$ such that $x = u - v$. Hence, $\mathbf{n}(x) = 0$ and thus $x = 0$ because $\mathbf{n}$ is a norm.
    
    If $\lambda > 0$ then
    $$\tilde{\mathbf{n}}(\lambda x) = \inf_{\lambda x = u - v}(\mathbf{n}(u) + \mathbf{n}(v)) = \inf_{x = u - v}(\mathbf{n}(\lambda u) + \mathbf{n}(\lambda v)) = \lambda \tilde{\mathbf{n}}(x).$$
    Moreover, it is easily seen that $\tilde{\mathbf{n}}(x) = \tilde{\mathbf{n}}(-x)$ for all $x \in \spn(F)$. Hence, if $\lambda \leq 0$ then $\tilde{\mathbf{n}}(\lambda x) =\tilde{\mathbf{n}}(|\lambda|(-x)) = |\lambda|\tilde{\mathbf{n}}(-x) = |\lambda|\tilde{\mathbf{n}}(x)$.
    
    Now we show that $\tilde{\mathbf{n}}$ satisfies the triangle inequality. If $y \in \spn(F)$ and $\epsilon > 0$ there exist $x_{1}, x_{2}, y_{1}, y_{2} \in F$ with $x = x_{1} - x_{2}$, $y = y_{1} - y_{2}$ and
    $$\mathbf n(x_{1}) + \mathbf n(x_{2}) < \tilde{\mathbf{n}}(x) + \epsilon \quad \text{ and } \quad \mathbf n(y_{1}) + \mathbf n(y_{2}) < \tilde{\mathbf{n}}(y) + \epsilon.$$
    Since $x + y = (x_{1} + y_{1}) - (x_{2} + y_{2})$, by definition of $\tilde{\mathbf{n}}$,
    $$\tilde{\mathbf{n}}(x + y) \leq \mathbf n(x_{1} + y_{1}) + \mathbf n(x_{2} + y_{2}) \leq \mathbf n(x_{1}) + \mathbf n(x_{2}) + \mathbf n(y_{1}) + \mathbf n(y_{2}) < \tilde{\mathbf{n}}(x) + \tilde{\mathbf{n}}(y) + 2\epsilon.$$
    Since $\epsilon$ may be arbitrarily small, we conclude that
    $$\tilde{\mathbf{n}}(x + y) \leq \tilde{\mathbf{n}}(x) + \tilde{\mathbf{n}}(y).$$
    This shows that $\tilde{\mathbf{n}}$ is a norm. 
\end{proof}

\begin{Theorem}\label{Theorem: normed cones quotient topology}
    Let $F$ be a real linear cone. Let $\mathbf n$ be a norm on $\spn(F)$, and endow $F$ with the subspace topology of the norm topology induced by $\mathbf n$, with which it becomes a topological cone. Assume that $\mathrm{Int}_{\tau_{\mathbf n}}(F) \neq \emptyset$. Then $\tau_{univ} = \tau_{quot}$ and both are induced by $\tilde{\mathbf{n}}$.
\end{Theorem}
\begin{proof}
    Let $\tau_{\tilde{\mathbf{n}}}$ denote the topology induced by $\tilde{\mathbf{n}}$ on $\spn(F)$. In accordance with our earlier conventions, we denote $V_{\tau_{\tilde{\mathbf{n}}}} = (\spn(F), \tau_{\tilde{\mathbf{n}}})$. From Corollary $\ref{Corollary: univeral topology finest}$ and Proposition $\ref{prop: quot topology too fine}$ we have
    $$\tau_{\tilde{\mathbf{n}}} \subseteq \tau_{univ} \subseteq \tau_{quot}.$$
    By the argument of Lemma \ref{Lemma: equivilence of norm with extension norm} below, $\mathbf n$ and $\tilde{\mathbf{n}}$ are equivalent norms. The difference map
    $$q: F \times F \longrightarrow (\spn(F), \tau_{\tilde{\mathbf{n}}}),
    \qquad q(u,v) = u-v$$
    is easily seen to be continuous, open and surjective, hence a quotient map. Therefore $\tau_{quot} = \tau_{\tilde{\mathbf{n}}}$, as desired.
\end{proof}

\begin{Lemma}\label{Lemma: monotone norm implies norm can be extended}
    In the setting of Proposition \ref{Prop: norm extension}, $\mathbf{n}$ agrees with $\tilde{\mathbf{n}}$ on $F$.
\end{Lemma}
\begin{proof}
    Let $x \in F$. By the definition of $\tilde{\mathbf n}$, $\tilde{\mathbf{n}}(x) \leq \mathbf{n}(x) + \mathbf{n}(0) = \mathbf{n}(x)$. On the other hand, it is always true that $\mathbf n(x) \leq \mathbf{\tilde{n}}(x)$ for all $x \in \spn(F)$.
\end{proof}

\section{Positively polarizable hyperbolic norms}\label{Section: Positively polarizable hyperbolic norms}
\subsection{Positively polarizable hyperbolic norms, induced inner product}
Henceforth, we work with the totally ordered field $\R$. Unless otherwise stated, \emph{cone} shall mean proper linear cone over $\R$ and \textit{vector space} shall mean vector space over $\R$. If $V$ is a vector space, an \textit{inner product} on $V$ is a (not necessarily positive definite) symmetric bilinear form.

\begin{Definition}[Hyperbolic norms]
Let $F$ be a cone. A \emph{hyperbolic norm} on $F$ is a map $\|\cdot\|: F \to [0,\infty]$ such that for all $v,w \in F$ and $\lambda \geq 0$:
\begin{enumerate}
    \item $\|v + w\| \geq \|v\| + \|w\|$,
    \item $\|\lambda v\| = \lambda \|v\|$.
\end{enumerate}
Here, we employ the convention $0 \cdot \infty:= 0$. We call $(F,\|\cdot\|)$ a \emph{hyperbolically normed cone}.
\end{Definition}

\begin{Example}[Examples of hyperbolic norms]
\begin{enumerate}
    \item[]
    \item Let $p \in [1,\infty)$. Consider the cone
    \begin{align*}
        F:=\{(x_0,x) \in \R \times \R^n : x_0 \geq |x|_p\},
    \end{align*}
    where $|x|_p$ is the $p$-norm of $x \in \R^n$. Clearly, $\spn(F) = \R \times \R^n$. The $p$-\emph{hyperbolic norm} on $F$ is defined via
    \begin{equation}
    \|(x_0,x)\|_p:=\left( |x_0|^p - \sum_{j=1}^n |x_j|^p \right)^{1/p}.
    \end{equation}
    For $p=2$, this is precisely the hyperbolic norm induced by the Minkowski metric $dx_0^2 - \sum_{j=1}^n dx_j^2$.
    \item Let $(X,\Omega,\mu)$ be a measure space and $0 \neq q < 1$. Let $F:=\{f \in L^0(\mu): f \geq 0\}$, where $L^0(\mu)$ are the $\mu$-measurable functions $f:X \to \R$ (up to $\mu$-a.e.\ equality). It is easily seen that $\spn(F) = L^0(\mu)$. On $F$, the hyperbolic $L^q$-norm is defined by
    \begin{equation}
        \|f\|_{L^q}:=\left(\int_X |f|^q \, d\mu\right)^{1/q}.
    \end{equation}
    We refer to Gigli \cite{gigli2025hyperbolic} for a detailed discussion concerning hyperbolic $L^q$-spaces.
\end{enumerate}
\end{Example}

\begin{Definition}[Positive polarizability]\label{Definition: positivepolarizability}
A hyperbolic norm $\|\cdot\|$ on a cone $F$ is said to be \emph{positively polarizable} if for all $v,w \in F$:
\begin{equation}
    \|v + 2w\|^2 + \|v\|^2 = 2 \|v + w\|^2 + 2\|w\|^2.
\end{equation}
\end{Definition}

\begin{Proposition}[Positive polarizability and induced inner product]
    A finite-valued hyperbolic norm $\|\cdot\|$ on a cone $F$ is positively polarizable if and only if 
    \begin{equation}
        \langle v,w \rangle := \frac{1}{2}(\|v + w\|^2 - \|v\|^2 - \|w\|^2)
    \end{equation}
    defines a nonnegatively bilinear and symmetric form $F \times F \to \R$, i.e., $\langle v,w \rangle = \langle w,v \rangle$ and $\langle \lambda v + \mu w,u\rangle = \lambda \langle v,u\rangle + \mu\langle w,u\rangle$, for all $v,w,u \in F$ and $\lambda,\mu \geq 0$. In this case, $\langle \cdot,\cdot\rangle$ extends uniquely to an indefinite inner product on $\spn(F)$.
\end{Proposition}
\begin{proof}
    This is proven in \cite[Lem.\ A.12]{beran2024nonlinear}.
\end{proof}

\begin{Lemma}[Elementary properties of the induced inner product]\label{Lemma: inner product properties}
Let $(F,\|\cdot\|)$ be a positively polarizable hyperbolically normed cone with $\|v\| < \infty$ for all $v \in F$ and let $\langle \cdot,\cdot \rangle$ denote the indefinite inner product induced on $\spn(F)$ by $\|\cdot\|$. Then
\begin{enumerate}
    \item[(i)] For all $v \in F$: $\|v\|^2 = \langle v,v\rangle$.
    \item[(ii)] (Reverse Cauchy-Schwarz inequality): For all $v,w \in F$: $\langle v,w\rangle \geq \|v\| \|w\|$. Moreover, $\langle v,w\rangle = \|v\| \|w\|$ if and only if $\|v + w\| = \|v\| + \|w\|$.
\end{enumerate}
\end{Lemma}
\begin{proof}
    \begin{enumerate}
        \item[]
        \item[(i)] This is elementary:
        \begin{align*}
            2\langle v,v\rangle = \|2v\|^2 - 2\|v\|^2 = 2 \|v\|^2.
        \end{align*}
        \item[(ii)] This is a consequence of the reverse triangle inequality for $\|\cdot\|$:
        \begin{align*}
            2\langle v,w\rangle = \|v + w\|^2 - \|v\|^2 - \|w\|^2 \geq (\|v\| + \|w\|)^2 - \|v\|^2 - \|w\|^2 = 2 \|v\| \|w\|.
        \end{align*}
        Clearly, equality holds if and only if $\|v + w\| = \|v\| + \|w\|$.
    \end{enumerate}
\end{proof}

\begin{Theorem}[Non-degeneracy of the induced inner product]\label{Theorem: innerproductnondegenerate}
Let $(F,\|\cdot\|)$ be a proper positively polarizable hyperbolically normed cone satisfying the following conditions:
\begin{enumerate}
    \item[(i)] (Finiteness) $\|v\| < +\infty$ for all $v \in F$.
    \item[(ii)] (Non-triviality) There exists $t \in F$ such that $\|t\| > 0$.
    \item[(iii)] (Strict hyperbolicity) Whenever $\|u + v \| = \|u \| + \|v\|$ for $u,v \in F$, then $u$ and $v$ are linearly dependent in $F$, i.e., $u = \lambda v$ or $v = \lambda u$ for some $\lambda \geq 0$.
\end{enumerate}
Then the inner product $\inner{\cdot}{\cdot}$ on $\spn(F)$ induced by $\norm{\cdot}$ is nondegenerate.
\end{Theorem}
\begin{proof}
    Suppose that $v \in \spn(F)$ is such that $\inner{v}{u} = 0$ for all $u \in \spn(F)$. We show that $v = 0$.	Write $v = x - y$ for some $x, y \in F$. Here $F$ is identified with its image in $\spn(F)$. If $\norm{x} = 0$ or $\norm{y} = 0$ then we let $x' = x + t$ and $y' = y + t$ where $t \in F$ is an element with $\norm{t} > 0$ whose existence is guaranteed by $(ii)$. Then $v = x' - y'$ and 
    $$\norm{x'} \geq \norm{x} + \norm{t} > 0.$$
    Similarly $\norm{y'} > 0$. Thus we can assume without loss of generality that $\norm{x}, \norm{y} > 0$. We have
    $$0 = \inner{v}{u} = \inner{x}{u} - \inner{y}{u},$$
    so $\inner{x}{u} = \inner{y}{u}$ for all $u \in \spn(F)$. Taking $u = x$ and then $u = y$ we get
    $$\norm{x}^{2} = \inner{y}{x} = \inner{x}{y} = \norm{y}^{2},$$
    so $\norm{x} = \norm{y}$. Since
    $$\inner{x}{y} = \norm{x}^{2} = \norm{x}\norm{y},$$
    by Lemma \ref{Lemma: inner product properties}(\textit{ii}),
    $$\norm{x + y} = \norm{x} + \norm{y}.$$
    By Assumption (\textit{iii}), we get that $x$ and $y$ are linearly dependent in $F$. Since $x,y \neq 0$, this means that there is $\lambda > 0$ such that $x = \lambda y$. Hence, $\norm{x} = \norm{\lambda y} = \lambda \norm{y}$. Thus $\lambda = 1$, $x = y$ and $v = x - y = 0$.
\end{proof}

\subsection{Lorentzianity of the induced inner product}\label{Subsection: Lorentzianityofinnerproduct}

For the remainder of the paper, we work with a fixed hyperbolically normed cone $(F, \norm{\cdot})$ satisfying the conditions of Theorem $\ref{Theorem: innerproductnondegenerate}$ (i.e., a finite, non-trivial, strictly hyperbolic, proper, positively polarizable hyperbolically normed cone). Fix $t_{*} \in F$ with $\inner{t_{*}}{t_{*}} = 1$. We have shown in Theorem \ref{Theorem: innerproductnondegenerate} that, under these conditions, the induced inner product $\inner{\cdot}{\cdot}$ on $\spn(F)$ is nondegenerate.

\begin{Definition}[Lorentzian inner products]\label{Definition: Lorentzianinnerproducts}
Let $V$ be a vector space of dimension $\dim V \geq 1$. We say that an inner product $\inner{\cdot}{\cdot}$ on $V$ is \emph{Lorentzian} if there exists a direct sum decomposition $V \cong V_1 \oplus V_2$, with $V_1$ one-dimensional, such that the restriction of $\inner{\cdot}{\cdot}$ to $V_1 \times V_1$ is positive definite and the restriction to $V_2 \times V_2$ is negative definite. In this case, we call the pair $(V, \inner{\cdot}{\cdot}$) a \emph{Lorentzian vector space}.
\end{Definition}

Note that Lorentzian inner products are automatically nondegenerate. If $\dim(V) < +\infty$, then Lorentzian inner products are precisely those with signature $(+,-,\dots,-)$. Both of these well-known facts also follow from the following proposition, which does not assume that $V$ has finite dimension.

\begin{Proposition}[Lorentz decomposition]\label{Prop: lorentz decomposition}
    Let $(V,\inner{\cdot}{\cdot})$ be a Lorentzian vector space. Then
    \begin{enumerate}
        \item There exists $t \in V$ with $\inner{t}{t} = 1$.
        \item For any $t \in V$ with $\inner{t}{t} > 0$, we have $V = \R t \oplus (\R t)^{\perp}$. Moreover, $\inner{\cdot}{\cdot}$ is positive definite on $\R t$ and negative definite on $(\R t)^{\perp}$. 
    \end{enumerate}
    We call a decomposition of this form with $\inner{t}{t} = 1$, a \textit{Lorentz} decomposition of $V$. Thus, this result says that every Lorentzian vector space has a Lorentz decomposition.
\end{Proposition}
\begin{proof}
    Write $V = V_{1} \oplus V_{2}$ as in Definition \ref{Definition: Lorentzianinnerproducts}. Pick any $x \in V_{1} \setminus \{0\}$ and set $t = x/\sqrt{\inner{x}{x}}$. Then $\inner{t}{t} = 1$. Since $\inner{\cdot}{\cdot}$ is non-degenerate on $\R t$, the sum $\R t + (\R t)^{\perp}$ is direct. It is clear that every $v \in V$ can be written in this way, as $w:= v - \inner{v}{t}t \in (\R t)^\perp$. It remains to show that $\inner{\cdot}{\cdot}$ is negative definite on $(\R t)^{\perp}$. Let $w \in (\R t)^{\perp}$. From the original decomposition, we can write $w = v_{1} + v_{2}$. Since $\dim V_{1} = 1$, $v_{1} = \alpha t$ for some $\alpha \in \R$. Then,
    $$0 = \inner{w}{t} = \inner{\alpha t + v_{2}}{t} = \alpha + \inner{v_{2}}{t}$$ Thus, $\inner{v_{2}}{t} = -\alpha$. Next,
    $$\inner{w}{w} = \inner{\alpha t + v_{2}}{\alpha t + v_{2}} = \alpha^{2} + 2\alpha \inner{v_{2}}{t} + \inner{v_{2}}{v_{2}} = -\alpha^{2} + \inner{v_{2}}{v_{2}}$$
    Since $\inner{\cdot}{\cdot}$ is negative definite on $V_{2}$, $\inner{v_2}{v_2}\leq 0$ and thus $\inner{w}{w} \leq 0$. Further, if $\inner{w}{w} = 0$ then $\inner{v_{2}}{v_{2}} = \alpha^{2} \geq 0$. Since $\inner{\cdot}{\cdot}$ is negative definite on $V_{2}$, this would imply that $v_{2} = 0$. Consequently, $\alpha = 0$ and $w = \alpha t + v_{2} = 0$, concluding the proof that the restriction of $\inner{\cdot}{\cdot}$ to $(\R t)^\perp$ is negative definite.
\end{proof}

\begin{Definition}[Induced positive-definite inner product]
Let $(V, \inner{\cdot}{\cdot})$ be a Lorentzian vector space with decomposition $V = \R t \oplus W$, so $\inner{t}{t} = 1$ and $W = (\R t)^\perp$. We define the projections $\alpha: V \longrightarrow \R$ and $w: V \longrightarrow W$ by 
$$v = \alpha_{v}t + w_{v}$$
We define the \textit{induced positive-definite inner product} (or \textit{Wick rotation}) $(\cdot, \cdot): V \times V \longrightarrow \R$ by
\begin{equation}
(v, u) = \alpha_{v}\alpha_{u} - \inner{w_{v}}{w_{u}}.
\end{equation}
The \textit{induced norm} is 
\begin{equation}
\textbf{n}(v) = \sqrt{(v, v)} = \sqrt{\alpha_{v}^{2} - \inner{w_{v}}{w_{v}}}.
\end{equation}
\end{Definition}

\begin{Remark}
     The form $(\cdot, \cdot)$ in fact defines a positive definite inner product on $V$ because $V \cong \R \times W$ and $(\cdot, \cdot)$ is the product of the positive definite inner products $(\alpha, \beta) \mapsto \alpha \beta$ and $-\inner{\cdot}{\cdot}$ in $\R$ and $W$, respectively. 
     %Note also that $\alpha = \inner{\cdot}{t}$.
\end{Remark}

\begin{Definition}[Causal sets]
    Let $V$ be a vector space and $\inner{\cdot}{\cdot}$ be an inner product on $V$. The \textit{set of causal vectors} is defined by
    $$C(V) = \{v \in V: \inner{v}{v} \geq 0\}.$$
    We may define $\norm{\cdot}: C(V) \longrightarrow [0, \infty)$ by $$\norm{v} = \sqrt{\inner{v}{v}}.$$
    For any $t \in V$, we define $C_{t}^{+}(V) = \{v \in C(V): \inner{v}{t} \geq 0\}$ and similarly $C_{t}^{-}(V) = \{v \in C(V): \inner{v}{t} \leq 0\}$. If $\inner{t}{t} > 0$ then $C_{t}^{+}(V)$ is called a \textit{(causal) future} and $C_t^-(V)$ a \textit{(causal) past}. 
\end{Definition}

\begin{Remark}
    Suppose $(F,\norm{\cdot})$ is a hyperbolically normed cone that is proper, finite, non-trivial, strictly hyperbolic and positively polarizable, and let $t_* \in F$ such that $\norm{t_*} = \inner{t_*}{t_*} = 1$ (recall the beginning of this section). Then we have that $F \subseteq C_{t_{*}}^{+}(\spn F)$ because $\norm{v} \geq 0$ and $\inner{v}{t_{*}} \geq \norm{v}\norm{t_{*}} \geq 0$ for all $v \in F$. Similarly $-F \subseteq C_{t_{*}}^{-}(\spn F)$.
\end{Remark}

The following lemma is very simple, but will be used repeatedly throughout the rest of the paper. 
\begin{Lemma}\label{Lemma: alpha positive on future}
     If $V$ is a Lorentzian vector space, $t \in V$ with $\inner{t}{t} = 1$, $V = (\R t) \oplus W$ the corresponding Lorentz decomposition, $\mathbf{n}$ the positive definite norm on $V$ arising from the Wick rotation and $x = \alpha_x t + w_x$ is in the future set $C_{t}^{+}(V)$, then $\mathbf{n}(w_{x}) \leq \alpha_{x}$.
\end{Lemma}

\begin{proof}
    Observe that
    $$0 \leq \inner{x}{x} = \alpha_{x}^{2} + \inner{w_{x}}{w_{x}} = \alpha_{x}^{2} - \mathbf{n}(w_{x})^{2}.$$
    Since $\alpha_{x} = \inner{x}{t} \geq 0$ and $\mathbf{n}(w_{x}) \geq 0$, we may rearrange and take square roots of both sides.
\end{proof}

\begin{Remark}
    In any inner-product space $(V, \inner{\cdot}{\cdot})$, future sets are invariant under positive scaling of $t$. That is, if $t \in V$ with $\inner{t}{t} > 0$ and $\lambda > 0$, we have $C^{+}_{t}(V) = C^{+}_{\lambda t}(V)$. 
\end{Remark}

\begin{Lemma}
    Let $V$ be a Lorentzian vector space. Then the reverse Cauchy-Schwarz inequality holds on any future set. That is,
    $$\inner{u}{v} \geq \norm{u}\norm{v}$$
    for all $u, v \in C_t^{+}(V)$ for any $t \in V$ with $\inner{t}{t} > 0$.
\end{Lemma}
\begin{proof}
    We may assume $\inner{t}{t} = 1$. Let $u, v \in C_{t}^{+}(V)$. We start by showing that $\inner{u}{v} \geq 0$. 
    Using the regular Cauchy-Schwarz inequality for $(\cdot, \cdot)$ and Lemma $\ref{Lemma: alpha positive on future}$,
    $$\inner{u}{v} = \alpha_{u}\alpha_{v} + \inner{w_{u}}{w_{v}} = \alpha_{u}\alpha_{v} - (w_{u}, w_{v}) \geq \alpha_{u}\alpha_{v} - \mathbf{n}(w_{u})\mathbf{n}(w_{v}) \geq 0.$$
    Thus, if $\norm{v} = 0$ then $\inner{u}{v} \geq \norm{u}\norm{v}$ is trivial. Assume $\norm{v} > 0$, let $v' = v/\norm{v}$ and let $u' = u - \inner{u}{v'}v'.$
    As seen in Proposition $\ref{Prop: lorentz decomposition}$, we can decompose $V$ as $V = \R v' \oplus (\R v')^{\perp}$. Since
    $$\inner{u'}{v'} = \inner{u - \inner{u}{v'}v'}{v'} = \inner{u}{v'} - \inner{u}{v'} = 0,$$
    we see that $u' \in (\R v')^{\perp}$ so $\inner{u'}{u'} \leq 0$. Expanding this out,
    \begin{align*}
        0 &\geq \inner{u'}{u'} \\
        &= \inner{u - \inner{u}{v'}v'}{u - \inner{u}{v'}v'} \\
        &= \inner{u}{u} - 2\inner{u}{v'}^{2} + \inner{u}{v'}^{2}\inner{v'}{v'}\\
        &= \norm{u}^{2} - \inner{u}{v'}^{2}
    \end{align*}
    Thus, $\inner{u}{v'}^{2} \geq \norm{u}^{2}$. Since $\norm{u} \geq 0$ and we've shown that $\inner{u}{v'} = \inner{u}{v}/\norm{v} \geq 0$, we can take square roots to get $\inner{u}{v'} \geq \norm{u}$. Finally, scaling both sides by $\norm{v}$ gives us the desired result.
\end{proof}

\begin{Proposition}[Future cones in Lorentzian vector spaces and their spans]\label{Prop: Future cones in Lorentzian vector spaces and their spans}
    If $V$ is a Lorentzian vector space and $C_{t}^{+}(V)$ is any future set then
    \begin{enumerate}
        \item $(C_{t}^{+}(V), \norm{\cdot}|_{C_{t}^{+}(V)})$ is a hyperbolically normed cone satisfying the conditions of Theorem $\ref{Theorem: innerproductnondegenerate}$. 
        \item $V = \spn(C_{t}^{+}(V))$.
    \end{enumerate}
\end{Proposition}
\begin{proof}
    Without loss of generality, assume that $\inner{t}{t} = 1$. We continue to use the notation from the previous results: So $V = (\R t) \oplus W$, and an element $x$ in $V$ is written $x = \alpha_x t + w_x$ according to this Lorentz decomposition. First, $0 \in C_{t}^{+}(V)$. If $v, -v \in C_{t}^{+}(V)$ then $\alpha_{v}, \alpha_{-v} = - \alpha_{v} \geq 0$. It follows that $0 \leq \inner{v}{v} = \inner{w_{v}}{w_{v}} \leq 0$ so by definiteness, $w_{v} = 0$. Thus $v = 0$, which shows that $C_t^+(V)$ is proper. Now we show that $\norm{\cdot}$ is hyperbolic and that $C_{t}^{+}(V)$ is convex. If $u, v \in C_{t}^{+}(V)$ then $\inner{u + v}{t} = \inner{u}{t} + \inner{v}{t} \geq 0$. Clearly, $C_{t}^{+}(V)$ is closed under non-negative scaling and
    $$\norm{u + v}^{2} = \inner{u + v}{u + v} = \norm{u}^{2} + \norm{v}^{2} + 2\inner{u}{v} \geq \norm{u}^{2} + \norm{v}^{2} + 2\norm{u}\norm{v} = (\norm{u} + \norm{v})^{2}$$
    shows both that $\norm{\cdot}$ is a hyperbolic norm and that $C_{t}^{+}(V)$ is closed under addition. We now show that $\inner{u}{v} = \norm{u}\norm{v}$ only when $u$ and $v$ are collinear. If $u = 0$ or $v = 0$, the claim is trivial, so we suppose $u, v \neq 0$ and $\inner{u}{v} = \norm{u}\norm{v}$. First, consider the special case when $\norm{u} = \norm{v} = 0$. Then $\inner{u}{v} = 0$. The equation $\norm{u}^{2} = 0$ implies that $\alpha_{u}^{2} = \mathbf{n}(w_{u})^{2}$. Since $u \in C_{t}^{+}(V)$, both of these quantities are non-negative. Thus $\alpha_{u} = \mathbf{n}(w_{u})$. Similarly, $\norm{v}^{2} = 0$ implies $\alpha_{v} = \mathbf{n}(w_{v})$. Finally, $\inner{u}{v} = 0$ gives $\alpha_{u}\alpha_{v} = (w_{u}, w_{v})$. Substituting the first two equations, we get $(w_{u}, w_{v}) = \mathbf{n}(w_{u})\mathbf{n}(w_{v})$. By the regular Cauchy-Schwarz inequality for $(\cdot, \cdot)$, we conclude that $w_{u}$ and $w_{v}$ are collinear. Hence, $w_{u} = \lambda w_{v}$ for some $\lambda \in \R$. Then
    $$\alpha_{u} = \mathbf{n}(w_{u}) = \mathbf{n}(\lambda w_{v}) = |\lambda|\mathbf{n}(w_{v}) = |\lambda|\alpha_{v}.$$
    It remains to be shown that $\lambda \geq 0$. If $\mathbf{n}(w_{v}) = 0$ then $\alpha_{u} = \alpha_{v} = 0$ and $w_{u} = w_{v} = 0$ so $u = v = 0$. Otherwise,
    $$\lambda = \frac{(w_{u}, w_{v})}{\mathbf{n}({w_{v}})^{2}} = \frac{\mathbf{n}(w_{u})\mathbf{n}({w_{v}})}{\mathbf{n}({w_{v}})^{2}} = \frac{\mathbf{n}({w_{u}})}{\mathbf{n}({w_{v}})} \geq 0.$$
    Thus, $u = \alpha_u t + w_u = \lambda (\alpha_v t + w_v) = \lambda v$. This concludes the case where $\norm{u} = \norm{v} = 0$. Without loss of generality, assume now $\norm{v} \neq 0$ (w.l.o.g. $\norm{v} = 1$). Define the Gram matrix
    $$G = \begin{pmatrix}
        \inner{u}{u} & \inner{v}{u} \\
        \inner{u}{v} & \inner{v}{v} \\
    \end{pmatrix}.$$
    Then by assumption
    $$\det(G) = \norm{u}^{2}\norm{v}^{2} - \inner{u}{v}^2 = 0.$$
    Hence there exists a nonzero $c = (c_{1}, c_{2}) \in \R^{2}$ such that $Gc = 0$. That is,
    $$\inner{c_{1}u +c_{2}v}{u} = 0, \text{ and }\inner{c_{1}u + c_{2}v}{v} = 0$$
    so $x = c_{1}u + c_{2}v$ is orthogonal to both $u$ and $v$. It follows that $\inner{x}{x} = 0$. Since $\norm{v} = 1$,  we may form the decomposition $\R v \oplus (\R v)^{\perp}$, and $x \in (\R v)^{\perp}$ where the inner product is negative definite. This would imply that $x = 0$ so $u$ and $v$ are collinear.
        
    To show that $\spn(C_{t}^{+}(V)) = V$, we aim to write any given $x \in V$ as $x = v_{1} - v_{2}$ where $v_{1}, v_2 \in C_{t}^{+}(V)$ are of the form
    $$v_{1} = \lambda t + \frac{1}{2}x \text{ and } v_{2} = \lambda t - \frac{1}{2}x$$
    for some $\lambda \in \R$. We must pick $\lambda$ such that the following four conditions hold:
    \begin{enumerate}
        \item $\norm{v_{1}}^{2} \geq 0 \implies \lambda^{2} + \lambda \alpha_{x} + \frac{1}{4}\inner{x}{x} \geq 0;$
        \item $\norm{v_{2}}^{2} \geq 0 \implies \lambda^{2} - \lambda \alpha_{x} + \frac{1}{4}\inner{x}{x} \geq 0;$
        \item $\inner{v_{1}}{t} \geq 0 \implies \lambda \geq -\frac{1}{2}\alpha_{x};$
        \item $\inner{v_{2}}{t} \geq 0 \implies \lambda \geq \frac{1}{2}\alpha_{x}.$
    \end{enumerate}
    It is easily seen that all four inequalities can be satisfied by choosing a sufficiently large $\lambda$. We conclude that $V = \spn(C_t^+(V))$.
\end{proof}

\begin{Definition}
    Let $V$ be a Lorentzian vector space. We will refer to the future and past sets $C^{+}_t(V)$ and $C^{-}_{t}(V)$ for any $\inner{t}{t} > 0$ as future and past \textit{cones}, respectively. This is justified by Proposition $\ref{Prop: Future cones in Lorentzian vector spaces and their spans}$.
\end{Definition}

\begin{Remark}
    Suppose $(N, \norm{\cdot})$ is any finite positively polarizable hyperbolically normed cone such that $\spn(N)$ with the inner product $\inner{\cdot}{\cdot}$ induced by $\norm{\cdot}$ is Lorentzian. Then $\norm{\cdot}$ is automatically non-trivial, so there exists some $t \in N$ with $\norm{t} = 1$. This implies that $N$ is contained in the future cone $C_{t}^{+}(\spn N)$ which in turn means that $\norm{\cdot}$ is also strictly hyperbolic. In particular, every cone whose span is Lorentzian, must be a subcone of a future cone.
\end{Remark}

For the purposes of the next result, we refer to a subset $S$ of a vector space $V$ as \textit{proper} if $S \cap (-S) = \{0\}$.

\begin{Proposition}
\label{Theorem: proper implies Lorentz}
    Let $V$ be a vector space with an inner product $\inner{\cdot}{\cdot}$. Then $V$ is Lorentzian if and only if $V$ contains a proper future set.
\end{Proposition}
\begin{proof}
    We already showed that every Lorentzian vector space contains a proper future cone (in fact, every future cone is proper). To show the converse, suppose $C_{t}^{+}(V)$ is a proper future set. Write $V = \R t + (\R t)^{\perp}$. Then $\inner{\cdot}{\cdot}$ is positive definite on $\R t$. This means the sum $ \R t \oplus (\R t)^{\perp}$ is direct. We show that $\inner{\cdot}{\cdot}$ is negative definite on $(\R t)^{\perp}$. If $w \in (\R t)^{\perp}$ then $\inner{w}{t} = 0$. If in addition, $\inner{w}{w} \geq 0$ then $w \in C_{t}^{+}(V)$. By the same logic, $-w \in C_{t}^{+}(V)$ so by assumption, $w = 0$. Thus, $\inner{\cdot}{\cdot}$ is negative definite on $(\R t)^{\perp}$.
\end{proof}

Note that in the above proof, we did not assume that the set $C_{t}^{+}(V)$ is a convex cone. We only assumed that $C_{t}^{+}(V) \cap C_{t}^{-}(V) = \{0\}$.

Let now $(F,\norm{\cdot})$ be an (abstract) hyperbolically normed cone satisfying the assumptions we fixed at the beginning of this section.

\begin{Lemma}\label{Lemma: positive on interior}
    If $\spn(F)$ (equipped with some topology) is a topological vector space, then $\norm{\cdot}$ is strictly positive on $\Int(F)$.
\end{Lemma}

\begin{proof}
    Suppose for contradiction that $v \in \Int(F)$ with $\norm{v} = 0$. If $v = \lambda t_{*}$ then $\norm{v} = \lambda = 0$ but $0 \notin \Int(F)$ (because $F$ is a proper cone). Thus $v$ is not a scalar multiple of $t_{*}$. It follows by strict hyperbolicity that
    $$\norm{v + t_{*}} > \norm{v} + \norm{t_{*}} = 1.$$
    Then
    $$2\inner{v}{t_{*}} = \norm{v + t_{*}}^{2} - \norm{v}^{2} - \norm{t_*}^2 = \norm{v + t_{*}}^{2} - 1 > 0.$$
    Finally, since $v \in \Int(F)$, we can take $\varepsilon > 0$ small enough so that $u = v - \varepsilon t_{*} \in F$ and that 
    $$\norm{u}^2 = \inner{u}{u} = \norm{v}^{2} - 2\varepsilon\inner{v}{t_{*}} + \varepsilon^{2}\norm{t_{*}}^{2} = -2\varepsilon \inner{v}{t_{*}} + \varepsilon^{2} < 0,$$
    which is a contradiction.
\end{proof}

\begin{Definition}[Core]
    The \textit{core} of a cone $F$ is defined as
    $$\core(F) = \{x \in F: \forall v \in \spn(F)\; \exists \, \epsilon_{v} > 0 : x + \epsilon_{v}v \in F\}.$$
\end{Definition}

We come to one of our main results of this section, where we give sufficient conditions under which the inner product on $\spn(F)$ induced by a sufficiently regular, positively polarizable hyperbolically normed cone $(F,\norm{\cdot})$ is Lorentzian.

\begin{Theorem}[Lorentzianity of the induced inner product] \label{Theorem: innerproductLorentzian}
    Let $(F,\norm{\cdot})$ be a finite, positively polarizable hyperbolically normed proper cone such that there exists $t_* \in F$ with $\norm{t_*} = 1$ and $\norm{\cdot}$ is strictly hyperbolic. Suppose, in addition, that $(F,\norm{\cdot})$ satisfies one of the following conditions:
    \begin{enumerate}
        \item[(i)] $F$ is a future cone in $\spn(F)$;
        \item[(ii)] There exists $t \in \core(F)$ with $\norm{t} = 1$;
        \item[(iii)] There exists a topology on $\spn(F)$ with which it is a TVS and $\inner{v}{v} = 0$ for all $v \in \partial F$.
    \end{enumerate}
    Then the induced inner product $\inner{\cdot}{\cdot}$ on $\spn(F)$ is Lorentzian.
\end{Theorem}
\begin{proof}
    Let us suppose $(i)$ holds, so $F = C_{t_*}^{+}(\spn F)$ for some $t_* \in \spn(F)$ with $\inner{t_*}{t_*} = 1$. Then $C_{t_*}^{+}(\spn F)$ is a proper cone and by Proposition $\ref{Theorem: proper implies Lorentz}$, $\spn(F)$ is Lorentzian. 
    
    Now, let us assume \textit{(ii)}. Let $\spn(F) = \R t \oplus (\R t)^{\perp}$ be an orthogonal decomposition of $\spn(F)$ and suppose there exists $0 \neq w \in (\R t)^{\perp}$ with $\inner{w}{w} \geq 0$. Let $U = \spn(t, w)$. Then $\inner{\cdot}{\cdot}|_{U \times U}$ is a positive definite inner product, so it induces a norm $\ell: U \longrightarrow \R$,
    $$\ell(u) = \sqrt{\inner{u}{u}}$$
    Note that if $u \in U \cap F$ then $\ell(u) = \norm{u}$. That is, $\ell$ and $\norm{\cdot}$ agree on $U \cap F$. Hence, $\norm{\cdot}$ satisfies both the triangle inequality and the reverse triangle on $U \cap F$, i.e.,
    $$\norm{u + u'} = \norm{u} + \norm{u'}$$
    for all $u, u' \in U \cap F$. By assumption of strict hyperbolicity, any such $u$ and $u'$ are linearly dependent. Taking $u' = t$, we see that $U \cap F \subseteq \R t$. There exists an $\epsilon > 0$ such that $t + \epsilon w \in F$. Since $t + \epsilon w \in U$, we get $t + \epsilon w \in U \cap F$ so $w \in \R t$, which implies $w = 0$, a contradiction. This shows the claim under the assumption \textit{(ii)}.
    
    Let us now assume $(iii)$. If $\spn(F)$ is a TVS and $\inner{v}{v} = 0$ for all $v \in \partial F$ then, in particular, $\norm{\cdot}$ is zero in $(\partial F) \cap F$. Lemma $\ref{Lemma: positive on interior}$ implies that 
    $$\Int(F) = \{v \in F: \norm{v} > 0\}$$
    Since $\Int(F)$ is open in $\spn(F)$ and $t \in \Int(F) \subseteq \core(F)$, condition \textit{(ii)} is satisfied.
\end{proof}

We close this subsection with a characterization of $F \subseteq \spn(F)$ as a future cone via self-duality.

\begin{Definition}[Dual cone]
    Let $(V,\inner{\cdot}{\cdot})$ be a Lorentzian vector space and $F \subseteq C(V)$ be a linear cone. The \textit{dual} of $F$ is 
    $$F^{*} = \{v \in C(V): \inner{v}{x} \geq 0,\;\forall x \in F\}$$
    $F$ is said to be \textit{self dual} if $F = F^{*}$.
\end{Definition}

\begin{Remark}
    Any cone $F$ as above is a subset of its dual $F^{*}$, which is itself a linear cone contained in $C(V)$. This follows directly from the reverse Cauchy-Schwarz inequality.
\end{Remark}

\begin{Proposition}[Future cones and self-duality]
    Let $(F,\norm{\cdot})$ be a positively polarizable, finite, proper, non-trivial hyperbolically normed cone such that the induced inner product on $\spn(F)$ is Lorentzian. Then $F \subseteq \spn(F)$ is a future cone if and only if $F$ is self-dual and there is $t \in \core(F)$ with $\norm{t} \neq 0$.
\end{Proposition}

\begin{proof}
    Suppose $F = C_{t}^{+}(\spn(F))$ for some $t \in \spn(F)=:V$ with $\inner{t}{t} = 1$, i.e., $F$ is a future cone. If $v \in \spn(F)$ and $\epsilon > 0$ then
    $$\inner{t + \epsilon v}{t + \epsilon v} = 1 + 2 \epsilon \alpha_{v} + \epsilon^{2}\inner{v}{v}$$
    and
    $$\inner{t + \epsilon v}{t} = 1 + \epsilon \alpha_{v}.$$
    We can take $\epsilon$ sufficiently small so that both of the above expressions are positive. Thus $t \in \core(F) \neq \emptyset$. Now we show that $F^{*} \subseteq F$. If $v \in F^{*}$, then $\alpha_{v} = \inner{v}{t} \geq 0$ by definition. We need to show that $\inner{v}{v} \geq 0$. Suppose for contradiction that $\inner{v}{v} < 0$. We construct an $x \in F$ with $\inner{v}{x} < 0$. Notice that
    $$0 > \inner{v}{v} = \alpha_{v}^{2} + \inner{w_{v}}{w_{v}} = \alpha_{v}^{2} - \mathbf{n}(w_{v})^{2},$$
    so $\mathbf{n}(w_{v}) > \alpha_{v}$. Set $w' = w_{v}/\mathbf{n}(w_{v})$ so that $\inner{w'}{w'} = -1$. Set $x = t + w'$. Then $\inner{x}{t} = 1$ and $\inner{x}{x} = 1 - 1 = 0$ so $x \in F$. Finally, 
    $$\inner{v}{x} = \inner{\alpha_{v}t + w_{v}}{t + w'} = \alpha_{v} + \inner{w_{v}}{w'} = \alpha_{v} + \mathbf{n}(w_{v})\inner{w'}{w'} = \alpha_{v} - \mathbf{n}(w_{v}) < 0.$$
    This contradicts $v \in F^{*}$. Thus $\inner{v}{v} \geq 0$ meaning that $v \in F$ and $F^{*} \subseteq F$, so $F$ is self-dual.
    
    Conversely, suppose $F = F^{*}$ and let $t \in \core(F)$ with $\norm{t} > 0$. We claim that $F = C_{t}^{+}(V)$. If $v \in F$ then $\inner{v}{v} = \norm{v}^{2} \geq 0$ and reverse Cauchy-Schwarz gives $\inner{v}{t} \geq 0$. Thus $F \subseteq C_{t}^{+}(V)$. Conversely, let $v \in C_{t}^{+}(V)$. If $x \in F \subseteq C_{t}^{+}(V)$, since the reverse Cauchy-Schwarz inequality holds on any future cone, $\inner{v}{x} \geq \norm{v}\norm{x} \geq 0$. Thus $v \in F^{*} = F$ so $C_{t}^{+}(V) = F$.
\end{proof}

\subsection{Order completeness and Lorentz--Hilbertianity}

In this final subsection, we establish a link between two completeness notions on $\spn(F)$, assuming throughout that it arises as a Lorentzian vector space of a positively polarizable hyperbolically normed cone $(F,\norm{\cdot})$ satisfying the assumptions specified at the beginning of Subsection \ref{Subsection: Lorentzianityofinnerproduct} (recall that $F$ is always assumed to be a proper linear cone over $\R$). The first completeness notion is the usual one with respect to the positive definite Wick rotation inner product, while the other one is order-theoretic, using the fact that there is a natural partial order relation on $\spn(F)$.

\begin{Definition}[Lorentz--Hilbert space]\label{Def: Lorentz Hilbert}
    If $V$ is a Lorentzian vector space that is complete with respect to the induced form $(\cdot, \cdot)$ in some Lorentz decomposition, then we call $V$ a \textit{Lorentz--Hilbert space}.
\end{Definition}

\begin{Remark}[On Lorentz--Hilbert spaces]
\begin{enumerate}
    \item[]
    \item It is easily observed that if $V = \R t \oplus W = \R \overline{t} \oplus \overline{W}$ are two Lorentz decompositions of $V$, then the two resulting metric space structures on $V$ are isometric. Thus, the definition of Lorentz--Hilbert space is independent of any choice of Lorentz decomposition.
    \item If $V = \R t \oplus W$ is any Lorentz decomposition of $V$, then $V$ is a Lorentz--Hilbert space if and only if $W$ is a Hilbert space under the inner product $(\cdot, \cdot)|_{W \times W} = -\inner{\cdot}{\cdot}|_{W \times W}$.
\end{enumerate}
\end{Remark}

Recall that whenever $F$ is a proper linear cone over $\R$, there exists a reflexive, transitive and antisymmetric relation (i.e., a partial order) $\leq$ on $\spn(F)$ defined via $v \leq w$ if and only if $w - v \in F$.

\begin{Definition}[Sequential forward completeness]
    Let $V$ be a topological vector space and let $\leq$ be a partial order on $V$. We say that $V$ is \textit{sequentially forward complete} if every non-decreasing sequence $(v_{j})_{j \in \N}$ which is bounded above, converges in $V$.
\end{Definition}

The notion of \textit{sequential backward completeness} can be defined analogously. It corresponds precisely to sequential forward completeness of the reversed partial order.

\begin{Lemma}
    If $\spn(F)$ is Lorentzian then $(x, y) \geq 0$ for all  $x, y \in F$.
\end{Lemma}

\begin{proof}
    Let $x, y \in F \subseteq C_{t_*}^+(\spn(F))$, where $\|t_*\| = 1$ and $\spn(F) = (\R t_*) \oplus W$. By the regular Cauchy-Schwarz inequality and Lemma $\ref{Lemma: alpha positive on future}$
    $$\inner{w_{x}}{w_{y}} = - (w_{x}, w_{y})\leq |(w_{x}, w_{y})| \leq \mathbf{n}(w_{x})\mathbf{n}(w_{y}) \leq \alpha_{x}\alpha_{y}$$
    which implies that $(x, y) = \alpha_{x}\alpha_{y} - \inner{w_{x}}{w_{y}} \geq 0.$
\end{proof}

\begin{Corollary}\label{Cor: Wick norm is monotone}
    The norm $\mathbf{n}$ is monotone on $F$, with respect to the order induced by $F$. That is, whenever $x,y \in F$ such that $x \leq y$ then $\mathbf{n}(x) \leq \mathbf{n}(y)$.
\end{Corollary}
\begin{proof}
    Let $x, z \in F$. Then
    $$\mathbf{n}(x + z)^{2} = (x + z, x + z) = \mathbf{n}(x)^{2} + \mathbf{n}(z)^{2} + 2(x, z).$$
    Since $\mathbf{n}(z), (x, z) \geq 0$, we get $\mathbf{n}(x + z) \geq \mathbf{n}(x)$. Taking $z = y - x \in F$ gives us the desired result.
\end{proof}

\begin{Remark}
    Recall that, given a norm $\mathbf{n}$, we can construct the norm $\tilde{\mathbf{n}}$ on $\spn(F)$. We showed in Lemma $\ref{Lemma: monotone norm implies norm can be extended}$ that $\tilde{\mathbf{n}}|_{F} = \mathbf{n}$.
\end{Remark}

In the next result, the interior is taken with respect to the topology of the norm $\mathbf{n}$ (which is defined on all of $\spn(F)$). 
\begin{Lemma}\label{Lemma: equivilence of norm with extension norm}
    If $\Int(F) \neq \emptyset$ then $\tilde{\mathbf{n}}$ and $\mathbf{n}$ are equivalent as norms on $\spn(F)$.
\end{Lemma}
\begin{proof}
    For any $u, v \in F$ with $x = u - v$, we have $\mathbf n(x) \leq  \mathbf n(u) + \mathbf n(v)$. Taking the infimum over all such $u, v$, we get $\mathbf n(x) \leq \tilde{\mathbf{n}}(x)$. Now suppose that there exists some $s \in \Int(F)$. That is, there exists a $\delta > 0$ such that $v \in F$ for all $v \in \spn(F)$ with $\mathbf n(s - v) < \delta$. We can now show that $\tilde{\mathbf{n}}(x) \leq K \mathbf n(x)$ for all $x \in \spn(F)$ and some $K > 0$. If $\mathbf n(x)= 0$ then $x = 0$ and $\tilde{\mathbf{n}}(x) = 0$ so we are done. Otherwise, let $\epsilon = \frac{\delta}{2\mathbf n(x)}$, $y_{+} = s + \epsilon x$, $y_{-} = s - \epsilon x$. Since 
    $$\mathbf n(s - y_{+}) = \mathbf n(s - y_{-}) = \mathbf n(\epsilon x) = \frac{\delta}{2} < \delta,$$
    it follows that $y_{+}, y_{-} \in F$. Note that $y_{+} - y_{-} = 2\epsilon x$, so
    $$x = \frac{y_{+} - y_{-}}{2\epsilon}.$$
    Let $u = \frac{y_{+}}{2 \epsilon}$ and $v = \frac{y_{-}}{2\epsilon}$. This gives us $x = u - v$ with $u, v \in F$. By definition of $\tilde{\mathbf n}$,
    $$\tilde{\mathbf{n}}(x) \leq \mathbf n(u) + \mathbf n(v) = \frac{\mathbf n(y_{+}) + \mathbf n(y_{-})}{2\epsilon}.$$
    By the triangle inequality,
    $$\mathbf n(y_{+}) + \mathbf n(y_{-}) = \mathbf n(s + \epsilon x) + \mathbf n(s - \epsilon x) \leq 2\mathbf n(s) + 2\epsilon \mathbf n(x).$$
    Thus,
    $$\tilde{\mathbf{n}}(x) \leq \frac{\mathbf n(s) + \epsilon \mathbf n(x)}{\epsilon} = \frac{2\mathbf n(x) \mathbf n(s)}{\delta} + \mathbf n(x) = \Big(\frac{2\mathbf n(s)}{\delta} + 1\Big)\mathbf n(x).$$
    Taking $K = \frac{2\mathbf n(s)}{\delta} + 1$ completes the proof.
\end{proof}

\begin{Lemma}
    Let $\tau_{\mathbf{n}}$ be the topology induced by the norm $\mathbf{n}$. Suppose $\Int_{\tau_{\mathbf{n}}}(F) \neq \emptyset$. If $F$ is given the subspace topology then $\tau_{\mathbf{n}} = \tau_{univ}$.
\end{Lemma}
\begin{proof}
    By Theorem $\ref{Theorem: normed cones quotient topology}$, the universal topology $\tau_{univ}$ is induced by the norm $\tilde{\mathbf{n}}$. By Lemma $\ref{Lemma: equivilence of norm with extension norm}$,  $\tilde{\mathbf{n}}$ and $\mathbf{n}$ induce the same topology.
\end{proof}

Our last main theorem provides an equivalence between order-theoretic completeness and Lorentz--Hilbertianity. Similar results in the more general context of Lorentzian products over normed spaces are due to Gigli \cite{gigli2025hyperbolic}.

\begin{Theorem}[Order-completeness and Lorentz--Hilbertianity]\label{Theorem: Sequential foward completeness and Lorentz Hilbert}
Suppose that $\spn(F)$ is a Lorentzian vector space. If $\spn(F)$ is sequentially forward complete with respect to the topology induced by $\mathbf n$ and $\Int(F) \neq \emptyset$, then $\spn(F)$ is a Lorentz--Hilbert space. Conversely, if $\spn(F)$ is a Lorentz--Hilbert space then it is sequentially forward complete.
\end{Theorem}
\begin{proof}
    Suppose $\spn(F)$ is a Lorentzian vector space such that $\Int(F) \neq \emptyset$ in the topology induced by $\mathbf n$. Assume that $\spn(F)$ is forward complete and that $(x_{k}) \subseteq \spn(F)$ is an $\mathbf n$-Cauchy sequence. By Lemma $\ref{Lemma: equivilence of norm with extension norm}$, $(x_{k})$ is also $\tilde{\mathbf{n}}$-Cauchy. Thus we can find a subsequence $(y_{k})$ such that $\mathbf{\tilde{n}}(y_{k + 1} - y_{k}) < 2^{-k}$. That is, for each $k$, there are $u_{k}, v_{k} \in F$ with $y_{k + 1} - y_{k} = u_{k} - v_{k}$ and 
    $\mathbf n(u_{k}) + \mathbf n(v_{k}) < 2^{-k}$. This means that the series $\sum_{k} u_{k}$ and $\sum_{k} v_{k}$ are both absolutely convergent (with respect to both norms $\tilde{\mathbf{n}}$ and $\mathbf n$). Let $V_{k}, U_{k}$ be the partial sums
    $$V_{k} = \sum_{j = 0}^{k}v_{j}, \quad U_{k} = \sum_{j = 0}^{k}u_{j}.$$
    It follows that both $(V_{k})$ and $(U_{k})$ are $\mathbf n$-Cauchy and hence there exists some $M > 0$ such that $\mathbf n(V_{k}), \mathbf n(U_{k}) < M$ for all $k$. Let $s \in \Int(F)$ and let $\delta > 0$ be such that $B_{\delta}(s) \subseteq F$. Let $V'_{k} = \frac{\delta}{M}V_{k}$. Then
    $$\mathbf n(V'_{k}) = \frac{\delta}{M}\mathbf n(V_{k}) < \delta,$$
    so $s - V'_{k} \in B_{\delta}(s) \subseteq F$. Hence $V'_{k} \leq s$. It follows that $V_{k} \leq \frac{M}{\delta} s$. A similar argument shows that $U_{k} \leq \frac{M}{\delta} s$. Hence, $(V_{k})$ and $(U_{k})$ are non-decreasing, order-bounded sequences. By assumption, they converge to $V$ and $U$ respectively. Thus,
    $$y_{k+1} = y_{0} + \sum_{j = 0}^{k}(y_{j + 1} - y_{j}) = y_{0} + \sum_{j = 0}^{k}(u_{j} - v_{j}) = y_{0} + U_{k} - V_{k} \to y_{0} + U - V.$$
    Since the original sequence $(x_{k})$ is $\mathbf n$-Cauchy and has a convergent subsequence, it must itself be convergent. 

    Now we prove the converse. Suppose that $\spn(F)$ is a Lorentz-Hilbert space and let $(v_{k}) \subseteq \spn(F)$ be a non-decreasing sequence bounded above by some $y \in \spn(F)$. For each $k$, we have $v_{k + 1} - v_{k} \in F$. By Lemma $\ref{Lemma: alpha positive on future}$,
    $$\mathbf n(w_{v_{k + 1}} - w_{v_{k}}) \leq \alpha_{v_{k + 1}} - \alpha_{v_{k}}.$$
    This implies that $\alpha_{v_{k + 1}} \geq \alpha_{v_{k}}$ so $(\alpha_{v_{k}})$ is also a non-decreasing sequence. By the same logic, $v_{k} \leq y$ implies that $\alpha_{v_{k}} \leq \alpha_{y}$. Similarly, $\mathbf n(w_{y} - w_{v_{k}}) \leq \alpha_{y} - \alpha_{v_{k}}$. By the triangle inequality,
    $$\mathbf n(w_{v_{k}}) \leq \mathbf n(w_y) + \mathbf n(w_{y} - w_{v_k}) \leq \mathbf n(w_y) + \alpha_{y} - \alpha_{v_{k}} \leq \mathbf n(w_y) + \alpha_{y} - \alpha_{v_0}.$$
    Hence $(\mathbf n(w_{v_{k}}))$ is bounded above. Since $(\alpha_{v_{k}})$ is a bounded nondecreasing sequence of real numbers, it converges. Similarly, $(w_{v_{k}})$ is a Cauchy sequence because for all $k < j$ (we set $w_k:=w_{v_k}$ and $\alpha_k:=\alpha_{v_k}$)
    $$\mathbf n(w_{j} - w_{k}) = \mathbf n\left(\sum_{i = k}^{j - 1}(w_{i + 1} - w_{i})\right) \leq \sum_{i = k}^{j - 1} \mathbf n(w_{i + 1} - w_{i}) \leq \sum_{i = k}^{j - 1}(\alpha_{i + 1} - \alpha_{i}) = \alpha_{j} - \alpha_{k} \to 0.$$
    Since $(\spn(F),\mathbf n)$ is complete by assumption, $(w_{k})$ converges. Consequently, $(v_{k})$ converges.
\end{proof}

We remark that $\spn(F)$ being a Lorentz--Hilbert space is also equivalent with sequential backward completeness, since the latter is simply forward completeness with respect to the relation $x \leq y$ if and only if $x - y \in F$.

\begin{Remark}[Krein spaces]

It is noteworthy that Lorentz--Hilbert spaces in the sense of our Definition \ref{Def: Lorentz Hilbert} are examples of Krein spaces, which are more general examples of inner product spaces admitting a direct sum decomposition such that both summands are Hilbert spaces with respect to the (negative of the) restriction of the indefinite inner product. In that setting, what we call Wick rotation norm and denote by $\mathbf n$ is known as a $J$-norm. We refer to Sorjonen \cite{Pekka} (also the classical text of Bognár \cite{Bognar}) for more details.
\end{Remark}

\section{Outlook and open problems}\label{Section: Outlook}

In this article, we have investigated linear cones $F$ over totally ordered fields $\mathbb{K}$. We have shown that $F$ embeds into a unique $\mathbb{K}$-vector space satisfying a universal property, called $\spn(F)$. Moreover, if $F$ is topological, then there is a universal topology on $\spn(F)$. We have studied positively polarizable hyperbolic norms on linear cones over $\R$, and given sufficient conditions under which the induced inner product on $\spn(F)$ is Lorentzian. Finally, we have linked the completeness under the Wick rotation of this inner product with order-theoretic completeness.

There are many open problems in the context of hyperbolic norms (see the discussions in \cite{gigli2025hyperbolic}). A significant one is to link the notion of positive polarizability (Definition \ref{Definition: positivepolarizability}) and infinitesimal Minkowskianity introduced in \cite[Def.\ 1.4]{beran2024nonlinear}, where the latter is a polarization identity involving the maximal weak subslopes $|df|$ of functions which are order-monotone with respect to the order structure of $\spn(F)$ (see \cite[Sec.\ 3.3]{beran2024nonlinear}). If $\spn(F)$ is finite-dimensional, then the equivalence between these two competing notions of polarizability is implicitly provided by the compatibility result proven in \cite[Thm.\ A.2]{beran2024nonlinear}. However, in infinite dimensions, one would need to establish a relationship between common notions of differentials of functions (more specifically, the hyperbolic norm of such) and the maximal weak subslope. Moreover, of similar interest is the question of timelike curvature-dimension conditions holding on infinite dimensional $\spn(F)$ arising from hyperbolically normed cones. Also, a causal theoretic analysis of $\spn(F)$ (arising from $(F,\norm{\cdot})$) would be of importance (in finite dimensions, $\spn(F)$ is always globally hyperbolic, see \cite[Prop.\ A.10]{beran2024nonlinear}).

Let us also note that hyperbolic norms could provide a way to study Lorentz--Finsler structures on manifolds which are nonsmooth, simply as maps on the tangent bundle whose restriction to each tangent space (or suitable tangent cones) is a hyperbolic norm. Such definitions of lower regularity Finsler metrics are customary in positive definite signature. This type of definition would likely put low regularity Finsler spacetimes directly into the framework of cone structures due to Minguzzi \cite{minguzzi2019causality}, whose theory is well-developed.

\section*{Data availability statement}

There is no data associated with this manuscript.

\section*{Conflict of interest statement}

The authors confirm that no conflict of interest exists for this work.

\section*{Acknowledgments}

The second author is grateful for the hospitality of the Department of Mathematics of the University of Vienna during his research visit in the summer of 2025 during the course of the 2nd Workshop on "A new geometry for Einstein's theory of relativity and beyond". The authors would like to thank Nicola Gigli for his comments on a preliminary version of this manuscript.

This research was funded in part by the Austrian Science Fund (FWF) [Grant DOI\\10.55776/EFP6 and 10.55776/J4913]. For open access purposes, the authors have applied a CC BY public copyright license to any author accepted manuscript version arising from this submission.

%\printbibliography
\bibliography{bibliography} 
\bibliographystyle{acm}
\addcontentsline{toc}{section}{References}

\end{document}